\documentclass[a4paper,12pt]{article}
\usepackage[text={6.5in,8.4in},centering]{geometry}
\usepackage{graphicx,url,subfig}
\RequirePackage[OT1]{fontenc}
\RequirePackage{amsthm,amsmath,amssymb,amscd}
\RequirePackage[numbers]{natbib}
\RequirePackage[colorlinks,citecolor=blue,urlcolor=blue]{hyperref}
\usepackage{enumerate}
\usepackage{datetime}
\usepackage{etex}

\makeatother
\numberwithin{equation}{section}
\allowdisplaybreaks

\theoremstyle{plain}
\newtheorem{theorem}{Theorem}[section]
\newtheorem{corollary}[theorem]{Corollary}
\newtheorem{lemma}[theorem]{Lemma}
\newtheorem{proposition}[theorem]{Proposition}

\theoremstyle{definition}
\newtheorem{definition}[theorem]{Definition}

\newcommand{\E}{\mathbb{E}}
\newcommand{\W}{\dot{W}}
\newcommand{\ud}{\ensuremath{\mathrm{d}}}

\newcommand{\sgn}{\text{sgn}}

\newcommand{\Indt}[1]{1_{\left\{#1 \right\}}}

\newcommand{\Norm}[1]{\left|\left|  #1   \right|\right|}

\newcommand{\Itos}{It\^{o}'s }
\newcommand{\spt}[1]{\text{supp}\left(#1\right)}
\newcommand{\InPrd}[1]{\left\langle #1 \right\rangle}

\newcommand{\calB}{\mathcal{B}}

\newcommand{\calF}{\mathcal{F}}

\newcommand{\calK}{\mathcal{K}}

\newcommand{\calL}{\mathcal{L}}
\newcommand{\calM}{\mathcal{M}}
\newcommand{\calN}{\mathcal{N}}

\newcommand{\bbC}{\mathbb{C}}
\newcommand{\bbN}{\mathbb{N}}

\newcommand*{\one}{{{\rm 1\mkern-1.5mu}\!{\rm I}}}

\newcommand{\R}{\mathbb{R}}
\newcommand{\myEnd}{\hfill$\square$}

\DeclareMathOperator{\Lip}{\mathit{L}}
\DeclareMathOperator{\LIP}{Lip}
\DeclareMathOperator{\lip}{\mathit{l}}

\DeclareMathOperator{\Vip}{\overline{\varsigma}}
\DeclareMathOperator{\vip}{\underline{\varsigma}}

\newcommand{\Dxa}{{}_x D_\delta^a}

\newcommand{\lMr}[3]{\:{}_{#1} #2_{#3}}
\newcommand{\tlMr}[4]{\:{}^{\hspace{0.2em}#1}_{#2} \hspace{-0.1em}#3_{#4}}

\title{On comparison principle and strict positivity of solutions to
the nonlinear stochastic fractional heat equations
}
\author{
{\bf Le Chen\footnote{
Research supported by a fellowship from Swiss National Science Foundation.}} \; and {\bf Kunwoo Kim}
\\[1em]
University of Utah
\date{}
}

\begin{document}
\maketitle
\begin{center}
\begin{minipage}[rct]{5 in}
\footnotesize \textbf{Abstract:}
In this paper, we prove a sample-path comparison principle for the nonlinear stochastic fractional heat equation on $\R$ with measure-valued initial data.
% Following arguments by Conus, Joseph and Khoshnevisan \cite{ConusEct12Corr}, we show how close to zero the solution can be.
We give quantitative estimates about how close to zero the solution can be.
These results extend Mueller's comparison principle on the stochastic heat equation to allow more general initial data
such as the (Dirac) delta measure and measures with heavier tails than linear exponential growth at $\pm\infty$.
These results generalize a recent work by Moreno Flores \cite{Moreno14Pos},
who proves the strict positivity of the solution to the  stochastic heat equation
with the delta initial data.
As one application, we establish the {\it full intermittency} for the equation.
As  an intermediate step, we prove the H\"older regularity of the solution starting from measure-valued initial data, which generalizes, in some sense, a recent work by Chen
and Dalang \cite{ChenDalang13Holder}.

\vspace{2ex}

\vspace{2ex}
\vspace{4ex}
\end{minipage}
\end{center}

% \tableofcontents
\setlength{\parindent}{1.5em}

%%%%%%%%%%%%%%%%%%%%%%%%%%%%%%%%%%%%%%%%%%%%%%
%%%%% MAIN: The chapters of the thesis
%%%%%%%%%%%%%%%%%%%%%%%%%%%%%%%%%%%%%%%%%%%%%%

% \mainmatter
% \graphicspath{{../figs/}}

\section{Introduction}

The comparison principle for differential equations tells us whether two solutions starting from two distinct initial conditions can compare with each other when the initial
conditions are comparable.
The {\it sample-path comparison principle} for stochastic differential equations (SDEs) and
also for stochastic partial differential equations (SPDEs) have been studied extensively;
see e.g. \cite[Chapter VI]{IkedaWatanabe89} and \cite[Chapter V. 40]{RogerWilliams00Vol2} for SDEs,
and \cite{Assing99Comparison,Kotelenez92Comparison,Milian02Comparison,Mueller91Support,Shiga94Two} for SPDEs.
A related problem is the {\it stochastic comparison principle}, which is of the form:
\[
\E\left(\Phi(u_t\right))\le \E\left(\Phi(v_t)\right),\quad\text{for all $t>0$,}
\]
where $\{u_t(x)\}$ and $\{v_t(x)\}$ solve SDEs or SPDEs, with the same initial data but comparable drift and diffusion coefficients. One looks for as large a class of functions $\Phi$ as
possible. See \cite{CoxFleischmannGreven96Comparison,Hajek85Mean,JackaTribe03Comparison,JosephDavarMueller14Strong}.

In this paper, we will focus on the pathwise comparison principle for
the following nonlinear stochastic fractional heat equation:
\begin{align}\label{E:FracHt}
 \begin{cases}
  \left(\displaystyle\frac{\partial}{\partial t} - \Dxa \right) u(t,x) =
\rho\left(u(t,x)\right) \dot{W}(t,x),& t\in \R_+^*:=\;]0,+\infty[\;,\: x\in\R\cr
u(0,\cdot) = \mu(\cdot)
 \end{cases}
\end{align}
where $a \in \;]1,2]$ is the order of the fractional differential
operator $\Dxa$ and $\delta$ ($|\delta|\le 2-a$) is its
skewness, $\dot{W}$ is the space-time white noise on $\R_+\times\R$, $\mu$ denotes the
initial data (a measure), and the function $\rho:\R\mapsto \R$ is Lipschitz
continuous. Throughout this paper, we assume that $a$ and $\delta$ are fixed constants such that
\begin{align}\label{E:aDelta}
a\in \;]1,2] \quad\text{and}\quad|\delta| \leq 2-a,
\end{align}
unless we state otherwise (see Corollary
\ref{C:FI}).

When $a=2$ and $\delta=0$, the fractional operator $\Dxa$ reduces to the Laplacian on $\R$,
which is the infinitesimal operator for a Brownian motion.
On the other hand, when $a\in\;]1,2[\;$ and $|\delta|\leq 2-a$,
the operator $\Dxa$ is the infinitesimal generator of an $a$-stable process with skewness $\delta$.
In particular, $\lMr{x}{D}{0}^a = -(-\Delta)^{a/2}$.
This fractional Laplace operator has been paid many attentions for several decades because of its non-local property, and thus it is widely used in many areas such as physics, biology, and finance to
model non-local (anomalous) diffusions.
We refer to \cite{MainardiEtc01Fundamental,UchaikinZolotarev99,Zolotarev86} for more details on these fractional operator and the related stable random variables.

The existence and uniqueness of a random field solution to \eqref{E:FracHt} have been studied
in \cite{ChenDalang13Heat,ChenDalang14FracHeat,ConusEtc11InitialMeasure, Debbi06Explicit,DebbiDozzi05On,FoondunKhoshnevisan08Intermittence}.
In particular, the existence, uniqueness, and moment estimates under measure-valued initial data have been established recently in
\cite{ChenDalang13Heat,ChenDalang14FracHeat,ConusEtc11InitialMeasure}.

% Before we talk about  comparison of solutions, we first need to specify the meaning of comparison between two initial measures.
% It is understood in the weak sense:
% Let $\mu_1$ and $\mu_2$ be two initial measures.
% We call that $\mu_1$ is {\it less than} than $\mu_2$, denoted as $\mu_1\le \mu_2$,  if for all nonnegative and continuous function with compact support $\phi\in C_c(\R)$,
% \begin{align}\label{E:CompMea}
%  \int_\R
% \phi(x-y)\mu_1(\ud y) \le \int_\R \phi(x-y)\mu_2(\ud y),\quad\text{for all $x\in\R$,}
% \end{align}
% and $\mu_1$ is {\it strictly less than} $\mu_2$, denoted as $\mu_1< \mu_2$, if in addition to the condition \eqref{E:CompMea}, for some $x\in\R$, the strict inequality holds in \eqref{E:CompMea}.

We now specify the \emph{weak} and \emph{strong} comparison principles.  Let $u_1(t,x)$ and $u_2(t,x)$ be two solutions to \eqref{E:FracHt} with initial measures $\mu_1$ and $\mu_2$,
respectively.
We say that \eqref{E:FracHt} satisfies the {\it weak comparison principle} if
$u_1(t,x)\le u_2(t,x)$ for all $t>0$ and $x\in\R$, a.s., whenever $\mu_1\le \mu_2$ (i.e., $\mu_2-\mu_1$ is a nonnegative measure).
And the equation \eqref{E:FracHt} is said to satisfy the {\it strong comparison principle} if $u_1(t,x)<u_2(t,x)$ for all $t>0$ and $x\in\R$, a.s., whenever $\mu_1<\mu_2$ (i.e., the
measure $\mu_2-\mu_1$ is nonnegative and nonvanishing).
Note that the stochastic comparison principle for \eqref{E:FracHt} with $\Phi(z)=|z|^k$ for $k \geq2$ (so-called the \emph{moment} comparison principle) has been shown
lately by Joseph, Khoshnevisan and Mueller
\cite{JosephDavarMueller14Strong}.

When $a=2$, the equation \eqref{E:FracHt} reduces to the stochastic heat equation (SHE).
The special case when $\rho(u)=\lambda u$ for some constant $\lambda \ne 0$ is called
the {\it parabolic Anderson model}; see
\cite{BertiniCancrini94Intermittence,CarmonaMolchanov94,FoondunKhoshnevisan08Intermittence}.
The weak comparison principle can be derived readily from the Feynman-Kac formula; see \cite{BertiniCancrini94Intermittence}.
But the proof of the strong comparison principle requires some more efforts.
This question is important because, e.g., the {\em Hopf-Cole solution} to the famous {\em Kardar-Parisi-Zhang equation (KPZ) } \cite{KPZ86} is the logarithm of the solution to the SHE.

For general $\rho$ which is Lipschitz continuous, we do not have the Feynman-Kac formula.
The weak comparison principle is no longer obvious.
In this case,
Mueller \cite{Mueller91Support} proves the strong comparison principle for the SHE on $\R$
for the initial data being absolutely continuous with respect to the Lebesgue measure with a bounded density function. Mueller uses the discrete Laplacian and  discretizes time to approximate the
solution to the SHE, which results in the weak comparison principle.
He then obtains the strong comparison principle by employing some large deviation estimates for the stochastic integral part of the solution.
Using Mueller's large deviation estimates, Shiga \cite{Shiga94Two} gives another proof of the strong comparison principle for the initial data being a so-called
$C_{\text{tem}}$ function, that is, a continuous function with both
tails growing no faster than $e^{\lambda|x|}$ for all $\lambda>0$.
He outlines a different approach for proving the weak comparison principle in the
appendix of his paper: smooth both the Laplace operator and the white noise so that one can apply the comparison principle for SDEs. We will follow his approach in our proof for the
weak comparison principle.

In both Mueller \cite{Mueller91Support} and Shiga \cite{Shiga94Two},
the initial data should be functions.
One natural question is whether the solution remains strictly positive if we run the system  \eqref{E:FracHt} starting from a measure, such as the Dirac delta measure.
Using the polymer model and following a convergence result by Alberts, Khanin and Questel \cite{AlbertEtc11},
Moreno Flores \cite{Moreno14Pos} recently  proved the strict positivity result for
the Anderson model (i.e., the case where $a=2$ and $\rho(u)=\lambda u$) with the delta initial data.
Our results below generalize their result to the stochastic fractional heat equation (i.e., $a\in \:]1,2]$), and moreover,
we consider general measure-valued initial data and allow $\rho$ to be any Lipschitz continuous function.

Recently, Conus, Joseph and Khoshnevisan \cite{ConusEct12Corr} give a more precise estimate on the strong comparison principle for the SHE. When the initial data is the Lebesgue measure, they prove
that for every $t>0$, there exist two finite constants $A>0$ and $B>0$ such that for all $\epsilon\in \:]0,1[\:$ and $x\in\R$,
\begin{align}\label{E:R-CJK}
P\left(u(t,x)<\epsilon\right)\le A \exp\left(-B \left[|\log(\epsilon)| \cdot\log\left(|\log(\epsilon)|\right)\right]^{3/2}\right).
\end{align}
Clearly, this result implies the strong comparison principle.
In \cite{MuellerNualart08}, Mueller and Nualart prove that when $a=2$ and the space domain is $[0,1]$ with the zero Dirichlet boundary condition, for some constants $C_0$ and $C_1$,
\begin{align}\label{E:R-MN}
P(u(t,x)<\epsilon)\le C_0\exp\left(-C_1 |\log\epsilon|^{3/2-\epsilon}\right).
\end{align}
We will generalize these results to the stochastic fractional heat equation \eqref{E:FracHt}
following  \cite{ConusEct12Corr}.
This shows how close to zero the solution to \eqref{E:FracHt} can be.

In order to state our results, we need some notation. Let
$\calM\left(\R\right)$ be the set of signed (regular) Borel measures on $\R$.
From the Jordan
decomposition, $\mu=\mu_+-\mu_-$ such that $\mu_\pm$ are two non-negative Borel
measures with disjoint support and denote $|\mu|=\mu_++\mu_-$.
As proved in \cite{ChenDalang14FracHeat},  the admissible initial data for \eqref{E:FracHt} is
\[
\calM_a\left(\R\right):=\left\{
\mu\in\calM(\R):\; \sup_{y\in\R}\int_\R |\mu|(\ud x) \frac{1}{1+|y-x|^{1+a}}<+\infty
\right\},\quad\text{for $a\in \:]1,2]$.}
\]
Moreover, when $a=2$, the admissible initial data can be more general than $\calM_2(\R)$: It can be any measures from the following set
\[
% \calM_H(\R):=
   \calM_H(\R) := \left\{\mu\in\calM(\R)\: :
\int_\R e^{-c x^2} |\mu|(\ud x)<+\infty,\: \text{for all $c>0$}\right\};
\]
see \cite{ChenDalang13Heat}.
Clearly, $\calM_a(\R)\subseteq \calM_H(\R)$.
In the following, a ``$+$" sign in the subscript means the subset of nonnegative
measures.
An important example in $\calM_{a,+}\left(\R\right)$ is the Dirac delta
measure.
For simplicity, denote
\[
\calM_a^*(\R):=\begin{cases}
\calM_a(\R)& \text{if $1<a<2$,}\cr
\calM_H(\R)& \text{if $a=2$}.
\end{cases}
\]

We will follow Shiga's arguments \cite{Shiga94Two} to prove the following weak comparison principle.
This result allows more general initial conditions than those in \cite{Mueller91Support} and \cite{Shiga94Two}.

\begin{theorem}[Weak comparison principle]
\label{T:WComp}
Let $u_1(t,x)$ and $u_2(t,x)$ be two solutions to \eqref{E:FracHt} with the initial data
$\mu_1$ and $\mu_2\in \calM_{a}^*(\R)$, respectively.
If $\mu_1\le \mu_2$, then
\begin{align}\label{E:Com}
P\left(u_1(t,x)\le u_2(t,x),\;\text{for all $t\ge 0$ and $x\in\R$}\right) =1.
\end{align}
% In particular, if $a=2$ (i.e. the stochastic heat equation), then for all initial data $\mu\in\calM_{H,+}(\R)$, \eqref{E:Com} holds.
\end{theorem}

Here is one example. Let $\delta_{z}$ be the Dirac delta function with unit mass at $x=z$.
Suppose that $\mu_1=\delta_0$ and $\mu_2=2 \delta_0$. Then $u_1(t,x)\le u_2(t,x)$ for all $t>0$ and $x\in\R$, a.s.

% \begin{remark}
As a direct consequence of Theorem \ref{T:WComp}, one can turn {\it weak intermittency} statements
in \cite{ChenDalang14FracHeat, FoondunKhoshnevisan08Intermittence} into the {\it full intermittency}.
More precisely, define the {\it upper and lower Lyapunov exponents of order $p$} by
\begin{align}\label{E:Lypnv-x}
\overline{m}_p(x) :=\mathop{\lim\sup}_{t\rightarrow+\infty}
\frac{1}{t}\log\E\left(|u(t,x)|^p\right),\quad
\underline{m}_p(x) :=\mathop{\lim\inf}_{t\rightarrow+\infty}
\frac{1}{t}\log\E\left(|u(t,x)|^p\right),
\end{align}
for all $p\ge 2$ and $x\in\R$.
According to Carmona and Molchanov \cite[Definition III.1.1, on p. 55]{CarmonaMolchanov94},
$u$ is {\it fully intermittent} if $\inf_{x\in\R} \underline{m}_2(x)>0$ and $m_1(x)\equiv 0$ for all $x\in\R$.
% \end{remark}

\begin{corollary}\label{C:FI}
Suppose that $a\in\:]1,2[\:$, $|\delta|< 2-a$ (strict inequality), $\mu\in\calM_{a,+}(\R)$, and $\rho$ satisfies that for some constants $\lip_\rho>0$ and $\vip\ge 0$, $\rho(x)^2\ge
\lip_\rho^2\left(\vip^2+x^2 \right)$ for all $x\in\R$.
If either $\mu\ne 0$ or $\vip\ne 0$, then the solution to \eqref{E:FracHt} is fully intermittent.
\end{corollary}
\begin{proof}
% [Proof of Corollary \ref{C:FI}]
By \cite[Theorem 3.4]{ChenDalang14FracHeat}, $\inf_{x\in\R} \underline{m}_2(x)>0$.
By Theorem \ref{T:WComp}, $u(t,x)\ge 0$ a.s. and so, $\E[|u(t,x)|]=\E[u(t,x)]=J_0(t,x)$ (see \eqref{E:J0}).
Therefore, $m_1(x)\equiv 0$ for all $x\in\R$.
\end{proof}

We adapt both Mueller and Shiga's arguments (see \cite{Mueller91Support, Shiga94Two}) to prove the strong comparison principle.
In the proof, following the idea of \cite[Theorem 5.1]{ConusEct12Corr}, we develop a large deviation result similar to \cite{Mueller91Support} using the Kolmogorov continuity theorem.

\begin{theorem}[Strong comparison principle]\label{T:SComp}
Let $u_1(t,x)$ and $u_2(t,x)$ be two solutions to \eqref{E:FracHt} with the initial data
$\mu_1$ and $\mu_2\in \calM_{a}^*(\R)$, respectively.
If $\mu_1< \mu_2$, then
\[
P\left(u_1(t,x)<u_2(t,x)\;\;\text{for all $t>0$ and $x\in\R$}\right) =1.
\]
\end{theorem}

The following theorem gives more precise information on the positivity of the solutions.
Let $\spt{f}$ denote the support of function $f$, i.e., $\spt{f}:=\left\{x\in\R: f(x)\ne 0\right\}$.

\begin{theorem}[Strict positivity]\label{T1:Rates}
Suppose $\rho(0)=0$ and let $u(t,x)$ be the solution to \eqref{E:FracHt} with the initial data $\mu\in\calM_a^*(\R)$. Then we have the following two statements:\\
(1) If $\mu>0$, then for any compact set $K\subseteq \R_+^*\times\R$,
there exist finite constants $A>0$ and $B>0$ which only depend on $K$ such that for small enough $\epsilon>0$,
\begin{align}\label{E:Rates1}
P\left(\inf_{(t,x)\in K}u(t,x) <\epsilon \right)&\le A \exp\left(-B |\log(\epsilon)|^{1-1/a} \log\left(|\log(\epsilon)|\right)^{2-1/a}\right).
\end{align}
(2) If  $\mu(\ud x)=f(x)\ud x$ with $f\in C(\R)$, $f(x)\ge 0$ for all $x\in\R$ and $\spt{f}\neq \emptyset$,
then for any compact set $D\subseteq\spt{f}$ and any $T>0$,
there exist finite constants $A>0$ and $B>0$ which only depend on $D$ and $T$ such that for all small enough $\epsilon>0$,
\begin{align}
P\left(\inf_{(t,x)\in \: ]0,T]\times D}u(t,x) <\epsilon\right)\le A \exp\left(-B \left\{ |\log(\epsilon)|\cdot \log\left(|\log(\epsilon)|\right)\right\}^{2-1/a}\right).
\end{align}
\end{theorem}

Theorem \ref{T1:Rates} shows that for all $t>0$, the function $x\mapsto u(t,x)$ does not have a compact support (see \cite{Mueller91Support} and
\cite[Section 6.3]{Mueller09Tools} for some other scenarios where the compact support property can be preserved). We also note that thanks to Theorem \ref{T1:Rates}, one can regard the solution $u(t,
x)$ to (1.1) as the density at location $x$ of a continuous particle system at time $t$, where particles move as independent $a$-stable processes but branch independently according to the noise term;
see \cite{JosephDavarMueller14Strong}.

% \begin{theorem}
% \label{T:Rates}
% Suppose $\rho(0)=0$ and let $u(t,x)$ be the solution to \eqref{E:FracHt} with initial data  $\mu(\ud x)=f(x)\ud x$ with $f\in L^{\infty}(\R)\cap C(\R)$. If $f(x)> 0$ for all $x\in\R$,
% then for any $T>0$ and $M>0$ there exist finite constants $A>0$ and $B>0$ which only depend on $T$ and $M$ such that for all small enough $\epsilon>0$,
% \begin{align*}
% P\big(u(t,x) <\epsilon \,\,\, &\text{for some $t\in\;]0,T[$ and some $x\inD\:$}\big)\\
% &\le A \exp\left(-B \left\{ |\log(\epsilon)|\cdot \log\left(|\log(\epsilon)|\right)\right\}^{(2a-1)/a}\right).
% \end{align*}
% \end{theorem}

Furthermore, Theorem \ref{T1:Rates} implies that for all compact sets
$K\subseteq\R_+^*\times\R$ and all $p>0$, the negative moments exists: $\E\left[|\inf_{(t,x)\in K} u(t,x)|^{-p}\right]<\infty$.
Note that part (2) of  Theorem \ref{T1:Rates} gives essentially the same rate as those in \eqref{E:R-CJK} and \eqref{E:R-MN} when $a=2$.

\bigskip

The next three theorems will be used in the proofs of the above theorems.
Since they are interesting by themselves, we list them below.

The first one, which is used in the  proof of Theorem \ref{T:WComp}, says that we can approximate a solution to
\eqref{E:FracHt} starting from $\mu \in \calM_a^*(\R)$ by a solution to \eqref{E:FracHt} starting from smooth initial conditions.
Define
\begin{align}\label{E:psi}
\psi_{\epsilon}(x)=
\begin{cases}
1&\text{if $|x|\le 1/\epsilon$},\cr
1+1/\epsilon-|x| &\text{if $1/\epsilon\le |x|\le 1+1/\epsilon$},\cr
0 &\text{if $|x|\ge 1+1/\epsilon$}.
\end{cases}
\end{align}

\begin{theorem}\label{T:Approx}
Suppose that $\mu\in\calM_a^*(\R)$.
Let $u(t,x)$ and $u_{\epsilon}(t,x)$ be the solutions to \eqref{E:FracHt} starting from
$\mu$ and
$((\mu\:\psi_\epsilon)*\lMr{\delta}{G}{a}( \epsilon,\cdot))(x)$, respectively. Then
\[
\lim_{\epsilon\rightarrow 0}\E\left[|u(t,x)-u_{\epsilon}(t,x)|^2\right] =0,\quad \text{for all $t>0$ and $x\in\R$.}
\]
\end{theorem}
% One interesting case is when one takes $g(\epsilon)\equiv +\infty$ and one only smooth the noise without truncating the tails of initial condition.

The following theorem, which is used in the proof of Theorem \ref{T:SComp}, shows the
sample-path regularity for the solutions to \eqref{E:FracHt}.
When the initial data has a bounded density, this has been proved in \cite{DebbiDozzi05On}.
For general initial data, the case where $a=2$ is proved in \cite{ChenDalang13Holder}.
The theorem below covers the cases where $1<a<2$.
We need some notation: Given a subset $K\subseteq \R_+\times\R$ and positive constants
$\beta_1,\beta_2$, denote by $C_{\beta_1,\beta_2}(K)$ the set of functions $v:
\R_+ \times \R \mapsto \R$ with the property that for each compact set
$D \subseteq K$, there is a finite constant $C$ such that for all $(t,x)$
and $(s,y)$ in $D$,
\[
   \vert v(t,x) - v(s,y) \vert \leq C \left[\vert t-s\vert^{\beta_1} + \vert x-y \vert^{\beta_2}\right].
\]
Denote
\[
   C_{\beta_1-,\beta_2-}(D) := \cap_{\alpha_1\in \;\left]0,\beta_1\right[} \cap_{\alpha_2\in \;\left]0,\beta_2\right[} C_{\alpha_1,\alpha_2}(D)\;.
\]

\begin{theorem}\label{T:Holder}
Let $u(t,x)$ be  the solution to \eqref{E:FracHt} starting from $\mu\in\calM_a^*(\R)$. 
Then we have 
\[
u\in
C_{\frac{a-1}{2a}-,\frac{a-1}{2}-}\left(\R_+^*\times\R\right),\;\;\text{a.s.}
\]
\end{theorem}

The last one, which is also used in the proof of Theorem \ref{T:SComp},
shows that the solution $u(t,x)$ to \eqref{E:FracHt} converges to the initial measure $\mu$
in the weak sense as $t\rightarrow 0$.
The case when $a=2$ is proved in \cite[Proposition 3.4]{ChenDalang13Holder}.
Let $C_c(\R)$ be the set of continuous functions with compact support
and $\InPrd{f,g}$ be the inner product in $L^2(\R)$.

\begin{theorem}\label{T:WeakSol}
Let $u(t,x)$ be  the solution to \eqref{E:FracHt} starting from $\mu\in\calM_a^*(\R)$. Then,
\[
\lim_{t\rightarrow 0}\InPrd{u(t,\circ),\phi} \stackrel{L^2(\Omega)}{=} \InPrd{\mu,\phi}\quad\text{for all $\phi\in C_c(\R)$.}
\]
\end{theorem}

In the following,
we first list some notation and preliminary results in Section \ref{S:Pre}.
Then we prove Theorem \ref{T:WComp} in Section \ref{S:WComp} with many
technical lemmas proved in the Appendix.
The proof of Theorem \ref{T:SComp} is presented in Section \ref{S:Scomp},
Theorem \ref{T1:Rates}  is proved in Section \ref{S:Rates}.
Finally, the three Theorems \ref{T:Approx}, \ref{T:Holder} and \ref{T:WeakSol} are proved in
Sections \ref{S:Approx}, \ref{S:Holder}, and \ref{S:WeakSol}, respectively.
% Some technical results are proved in Appendix.

% \newpage
\section{Notation and some preliminaries}\label{S:Pre}
The Green function associated to the
problem
\eqref{E:FracHt} is
\begin{align}\label{E:Green}
\lMr{\delta}{G}{a}(t,x) := \calF^{-1}
\left[\exp\left\{\lMr{\delta}{\psi}{a}(\cdot) t\right\} \right](x)
= \frac{1}{2\pi}\int_\R \ud \xi\: \exp\left\{i \xi x - t |\xi|^a
e^{-i \delta \pi\: \sgn(\xi)/2}\right\},
\end{align}
where $\calF^{-1}$ is the inverse Fourier transform and
\[
\lMr{\delta}{\psi}{a}(\xi) = -|\xi|^a e^{-i \delta \pi\:
\sgn(\xi)/2}\;.
\]
Denote the solution to the homogeneous equation
\begin{align}\label{E:FracHt-H}
 \begin{cases}
  \left(\displaystyle\frac{\partial}{\partial t} - \Dxa \right) u(t,x) = 0,&
t\in \R_+^*\;,\: x\in\R,\cr
u(0,\cdot) = \mu(\cdot),
 \end{cases}
\end{align}
by
\begin{align}\label{E:J0}
J_0(t,x) := \left(\lMr{\delta}{G}{a}(t,\cdot)*\mu\right)(x)
=\int_\R \mu(\ud y)\:\lMr{\delta}{G}{a}(t,x-y),
\end{align}
where ``$*$'' denotes the convolution in the space variable.

Following notation in \cite{ChenDalang13Heat}, let $W=\left\{
W_t(A),\, A\in\calB_b(\R),\, t\ge 0 \right\}$
be a space-time white noise
defined on a probability space $(\Omega,\calF,P)$, where
$\calB_b\left(\R\right)$ is the
collection of Borel sets with finite Lebesgue measure.
Let  $(\calF^0_t,\, t\ge 0)$ be the natural filtration generated by $W$ and augmented by the $\sigma$-field $\calN$ generated
by all $P$-null sets in $\calF$:
\[
\calF_t^0 = \sigma\left(W_s(A):0\le s\le
t,A\in\calB_b\left(\R\right)\right)\vee
\calN,\quad t\ge 0.
\]
Define $\calF_t := \calF_{t+}^0 = \wedge_{s>t}\calF_s^0$ for $t\ge 0$.
In the following, we fix this filtered
probability space $\left\{\Omega,\calF,\{\calF_t:t\ge0\},P\right\}$.
We use $\Norm{\cdot}_p$ to denote the
$L^p(\Omega)$-norm ($p\ge 1$).
% With this setup, $W$ becomes a worthy martingale measure in the sense of Walsh
% \cite{Walsh86}, and $\iint_{[0,t]\times\R}X(s,y) W(\ud s,\ud y)$ is
% well-defined in this reference for a suitable class of random fields
% $\left\{X(s,y),\; (s,y)\in\R_+\times\R\right\}$.

The rigorous meaning of the SPDE \eqref{E:FracHt} is the integral (mild) form
\begin{equation}
\label{E:WalshSI}
 \begin{aligned}
  u(t,x) &= J_0(t,x)+I(t,x),\quad\text{where}\cr
I(t,x) &=\iint_{[0,t]\times\R}
\lMr{\delta}{G}{a}\left(t-s,x-y\right)\rho\left(u(s,y)\right)W(\ud s,\ud
y),
 \end{aligned}
\end{equation}
where the stochastic integral is the Walsh integral \cite{Walsh86}.

\begin{definition}\label{DF:Solution}
A process $u=\left(u(t,x),\:(t,x)\in\R_+^*\times\R\right)$  is called a {\it
random field solution} to
\eqref{E:FracHt} if:
\begin{enumerate}[(1)]
 \item $u$ is adapted, i.e., for all $(t,x)\in\R_+^*\times\R$, $u(t,x)$ is
$\calF_t$-measurable;
\item $u$ is jointly measurable with respect to
$\calB\left(\R_+^*\times\R\right)\times\calF$;
\item $\left(\lMr{\delta}{G}{a}^2 \star \Norm{\rho(u)}_2^2\right)(t,x)<+\infty$
for all $(t,x)\in\R_+^*\times\R$,
where ``$\star$'' denotes the simultaneous
convolution in both space and time
variables. Moreover, the function
$(t,x)\mapsto I(t,x)$ mapping $\R_+^*\times\R$ into
$L^2(\Omega)$ is continuous;
\item $u$ satisfies \eqref{E:WalshSI} a.s.,
for all $(t,x)\in\R_+^*\times\R$.
\end{enumerate}
\end{definition}

Throughout the paper, we assume that the function $\rho:\R\mapsto \R$ is Lipschitz
continuous with Lipschitz constant $\LIP_\rho>0$, and moreover, for some constants $\Lip_\rho>0$ and $\Vip \ge 0$,
\begin{align}\label{E:LinGrow}
|\rho(x)|^2 \le \Lip_\rho^2 \left(\Vip^2 +x^2\right),\qquad \text{for
all $x\in\R$}.
\end{align}
Note that the above growth condition \eqref{E:LinGrow} is a consequence of $\rho$ being Lipschitz continuous.

Let $a^*$ be the dual of $a$, i.e., $1/a+1/a^*=1$.
The following constant is finite:
\begin{align}\label{E:Cst-dLa}
\Lambda=\lMr{\delta}{\Lambda}{a} := \sup_{x\in\R} \lMr{\delta}{G}{a}(1,x)\;,
\end{align}
and in particular, $\lMr{0}{\Lambda}{a}=
\pi^{-1}\Gamma\left(1+1/a\right)$; see \cite[(3.10)]{ChenDalang14FracHeat}.
In the following, we often omit the dependence of this constant on $\delta$ and
$a$ and simply write $\lMr{\delta}{\Lambda}{a}$ as $\Lambda$.
This rule will also apply to other constants.

For all $(t,x)\in\R_+^*\times \R$, $n\in\bbN$ and $\lambda\in\R$, define
\begin{align}
\notag
\calL_0\left(t,x;\lambda\right) &:= \lambda^2 \lMr{\delta}{G}{a}^2(t,x)  \\
\label{E:Ln}
\calL_n\left(t,x;\lambda\right)&:=
%\mathop{
\underbrace{\left(\calL_0\star \cdots\star\calL_0\right)}_{\text{$n+1$
factors } \calL_0(\cdot,\circ;\lambda)}
%}_{\calL_0(\cdot,\circ;\lambda)}(t,x)
,\quad\text{for $n\ge 1$,}\\
\calK\left(t,x;\lambda\right)&:= \sum_{n=0}^\infty
\calL_n\left(t,x;\lambda\right).
\label{E:K}
\end{align}
% For $t\ge 0$, define
% \[
% \calH(t;\lambda) := \left(1\star \calK(\cdot,\circ;\lambda)\right)(t,x).
% \]
We apply the following conventions to $\calK(t,x;\lambda)$:
\begin{align*}
 \calK(t,x) := \calK(t,x;\lambda),\quad
\overline{\calK}(t,x) := \calK\left(t,x;\Lip_\rho\right),\quad
\widehat{\calK}_p(t,x) := \calK\left(t,x;4\sqrt{p} \Lip_\rho\right),\quad\text{for $p\ge 2$}.
\end{align*}
%
%
% For $\tau\ge t> 0$ and $x, y\in\R$, define
% \begin{align*}
% \calI(t,x,\tau, y;\vv,\lambda) :=&
% \lambda^2\int_0^t \ud r\int_\R  \ud z
% \left[
% J_0^2(r,z)+\left(J_0^2(\cdot,\circ)\star
% \calK(\cdot,\circ;\lambda)\right)(r,y)+ \vv^2
% \left(\calH(r;\lambda)+1\right)
% \right]\\
% & \qquad\qquad \times \lMr{\delta}{G}{a}(t-r,x-z)\lMr{\delta}{G}{a}(\tau-r,y-z).
% \end{align*}
The following theorem is from \cite[Theorem 3.1]{ChenDalang14FracHeat} for $1<a<2$ and \cite[Theorem 2.4]{ChenDalang13Heat} for $a=2$.

\begin{theorem}[Existence,uniqueness and moments]\label{T:ExUni}
Suppose that $\mu\in \calM_a^*\left(\R\right)$, and $\rho$ is Lipschitz continuous and satisfies \eqref{E:LinGrow}.
Then the SPDE \eqref{E:FracHt} has a unique (in the sense of versions) random
field solution $\{u(t,x)\!: (t,x)\in\R_+^* \times \R \}$ starting from $\mu$. Moreover, for all even integers $p\ge 2$, all $t>0$ and $x\in\R$,
\begin{align}\label{E:MomUp}
 \Norm{u(t,x)}_p^2 \le
\begin{cases}
 J_0^2(t,x) + \left(\left[\Vip^2+J_0^2\right] \star \overline{\calK} \right)
(t,x),& \text{if $p=2$}\;,\cr
2J_0^2(t,x) + \left(\left[\Vip^2+2J_0^2\right] \star \widehat{\calK}_p \right)
(t,x),& \text{if $p>2$}\;.
\end{cases}
\end{align}
\end{theorem}

In order to use the moment bounds in \eqref{E:MomUp}, we need some estimates on $\calK(t,x)$.
Recall that if the partial differential operator is the heat operator
$\frac{\partial }{\partial  t}-\frac{\nu}{2} \Delta$ where $\nu>0$, then
\[
\calK^{\mbox{\scriptsize heat}}(t,x;\lambda) = G_{\frac{\nu}{2}}(t,x)
\left(\frac{\lambda^2}{\sqrt{4\pi\nu t}}+\frac{\lambda^4}{2\nu}
\: e^{\frac{\lambda^4 t}{4\nu}}\Phi\left(\lambda^2
\sqrt{\frac{t}{2\nu}}\right)\right),
\]
where $\Phi(x)$ is the distribution function of the standard normal
random variable and $G_\nu(t,x)=\frac{1}{\sqrt{2\pi\nu t}}\exp\left(-x^2/(2\nu t)\right)$; see \cite{ChenDalang13Heat}. If the partial
differential operator is the wave operator
$\frac{\partial ^2 }{\partial t^2}-\kappa^2 \Delta$ where $\kappa>0$, then
\[
\calK^{\mbox{\scriptsize wave}}(t,x;\lambda) = \frac{\lambda^2}{4}
I_0\left(\sqrt{\frac{\lambda^2((\kappa t)^2-x^2)}{2\kappa}}\right) \Indt{|x|\le
\kappa t},
\]
where $I_0(x)$ is the modified Bessel function of the first kind
of order $0$; see \cite{ChenDalang14Wave}.
Except these two cases, there are no explicit formulas for $\calK(t,x)$.
The following upper bound on $\calK(t,x)$ from \cite[Proposition 3.2]{ChenDalang14FracHeat} will be useful in this paper.

\begin{proposition}
\label{P:UpperBdd-K}
Let $\gamma:=\lambda^2\Lambda\:\Gamma(1-1/a)$.
For some finite constant $C=C(\lambda)>0$,
\begin{align}
\calK(t,x;\lambda)
&\le
\frac{C }{t^{1/a}} \lMr{\delta}{G}{a}(t,x)
\left(
1+t^{1/a} \exp\left(\gamma^{a^*} t\right)
\right), \quad\text{for all $t\ge 0$ and $x\in\R$.}
\label{E:UpBd-K}
\end{align}
\end{proposition}

% \section{Proofs of the main results}
% \label{S:Main}

\section{Proof of Theorem \ref{T:WComp}}\label{S:WComp}
Before proving Theorem \ref{T:WComp}, we need some preparation.
One may view $\lMr{\delta}{G}{a}(t,x)$ as an operator, denoted by $\lMr{\delta}{\bold G}{a}(t)$ for clarity, as follows:
\[
\lMr{\delta}{\bold G}{a}(t) f (x) := \left(\lMr{\delta}{G}{a}(t,\cdot)*f\right)(x).
\]
Let $\bold I$ be the identity operator: $\bold I f(x)=(\delta *f)(x) = f(x)$.
Set
\[
\tlMr{\epsilon}{\delta}{D}{a} = \frac{\lMr{\delta}{\bold G}{a}(\epsilon)-\bold I}{\epsilon}.
\]
Let
\begin{align}
\tlMr{\epsilon}{\delta}{\bold G}{a}(t) = \exp(t\tlMr{\epsilon}{\delta}{D}{a})&=
e^{-t/\epsilon}\sum_{n=0}^\infty \frac{(t/\epsilon)^n}{n!} \lMr{\delta}{\bold G}{a}(n\epsilon)=e^{-t/\epsilon}\bold I + \tlMr{\epsilon}{\delta}{\bold R}{a}(t),
\end{align}
where the operator $\tlMr{\epsilon}{\delta}{\bold R}{a}(t)$ has a density, denoted by $\tlMr{\epsilon}{\delta}{R}{a}(t,x)$, which is equal to
\begin{align}\label{E:R}
\tlMr{\epsilon}{\delta}{R}{a}(t,x)&=
e^{-t/\epsilon}\sum_{n=1}^\infty \frac{(t/\epsilon)^n}{n!} \lMr{\delta}{G}{a}(n\epsilon,x).
\end{align}
One may also write the kernel of $\tlMr{\epsilon}{\delta}{\bold G}{a}(t)$ as
\begin{align} \label{E:Geps}
\tlMr{\epsilon}{\delta}{G}{a}(t,x) = e^{-t/\epsilon} \delta_0(x) + \tlMr{\epsilon}{\delta}{R}{a}(t,x).
\end{align}
The two operators $\tlMr{\epsilon}{\delta}{\bold G}{a}$ and $\lMr{\delta}{\bold G}{a}$ are close in many senses;
see Appendix for more details.

\begin{proof}[Proof of Theorem \ref{T:WComp}]
Denote $\phi_\epsilon(x)=(2\pi\epsilon)^{-1/2}\exp\left(-x^2/(2\epsilon)\right)$.
For $\epsilon>0$ and $x\in\R$, denote
\[
W_x^\epsilon (t) := \int_0^t \int_\R \phi_\epsilon(x-y)W(\ud s,\ud y),\quad\text{for $t\ge 0$}.
\]
Clearly, $t\mapsto W_x^\epsilon (t)$ is a one-dimensional Brownian motion.
Denote $\W_x^\epsilon(t) = \frac{\ud }{\ud t}W_x^\epsilon (t)$. Then the quadratic variation of $\ud W_x^\epsilon(t)$ is
\begin{align}\label{E:Quad}
\ud \InPrd{W_x^\epsilon(t)} = \int_\R \phi_\epsilon^2(x-y)\ud y \ud t =\frac{1}{\sqrt{4\pi\epsilon}}\: \ud t.
\end{align}
Consider the following stochastic partial differential equation
\begin{align}\label{E:Ueps}
\begin{cases}
\displaystyle
 \frac{\partial }{\partial t} u_\epsilon(t,x) = \tlMr{\epsilon}{\delta}{D}{a} u_\epsilon(t,x)
+ \rho(u_\epsilon(t,x)) \W_x^\epsilon(t), & \quad t>0, \; x\in\R, \\[0.7em]
u_\epsilon(0,x)= (\mu*\lMr{\delta}{G}{a}(\epsilon,\cdot))(x), \quad x\in\R.
\end{cases}
\end{align}
Since $\rho$ is globally Lipschitz continuous, \eqref{E:Ueps} has a unique strong solution
\[
u_\epsilon(t,x)= (\mu*\lMr{\delta}{G}{a}(\epsilon,\cdot))(x)
+ \int_0^t \ud s \tlMr{\epsilon}{\delta}{D}{a} u_\epsilon(s,x) +
\int_0^t \rho(u_\epsilon(s,x))\ud W_x^\epsilon(s).
\]

{\vspace{0.7em}\noindent \bf Step 1.}
Let $u_{\epsilon,i}(t,x)$ be the solutions to \eqref{E:Ueps} with initial data $\mu_i$, $i=1,2$, respectively.
Denote $v_\epsilon(t,x):= u_{\epsilon,2}(t,x)-u_{\epsilon,1}(t,x)$. We will prove that
\begin{align}\label{E:NonUe}
P\left(v_\epsilon(t,x)\ge 0,\; \text{for every $t>0$ and $x\in\R$}\right) =1.
\end{align}
Let $a_n=-2(n^2+n+2)^{-1}$, $n\ge 0$. Then $a_n\uparrow 0$ as $n\rightarrow \infty$ and
$\int^{a_n}_{a_{n-1}}x^{-2}\ud x = n$. Let $\psi_n(x)$, $n=1,2,\dots$, be nonnegative continuous functions supported on $\:]a_{n-1},a_n[\:$ such that
\[
0\le \psi_n(x)\le \frac{2}{n x^2}\quad\text{and}\quad \int^{a_n}_{a_{n-1}}\psi_n(x) \ud x=1.
\]
Define
\[
\Psi_n(x) := \int_0^x\ud y\int_0^y \psi_n(z)\ud z.
\]
Clearly, $\Psi_n(x)\in C^2(\R)$ with $\Psi_n''(x)=\psi_n(x)$, $\Psi_n(x)=0$ for $x\ge 0$, and $-1\le \Psi_n'(x) =\int_0^x \psi_n(z)\ud z\le 0$ for all $x\in\R$. Let $\one(\cdot)$
denote the indicator function. Here are
three important properties: For all $x\in\R$, as $n\rightarrow+\infty$,
\begin{align}\label{E:PsiConv}
\Psi_n(x)\uparrow -(x\wedge 0)=:\Psi(x), \quad
\Psi_n'(x)\downarrow -\one (x<0)
\quad\text{and}\quad
\Psi_n'(x)\:x\uparrow \Psi(x),
\end{align}
Because for each $x\in\R$ fixed, $u_\epsilon(t,x)$ is a semi-martingale, by \Itos formula,
\begin{align*}
 \Psi_n(v_\epsilon(t,x))=& \int_0^t \Psi_n'(v_\epsilon(s,x)) \left[\rho(u_{\epsilon,2}(s,x))-\rho(u_{\epsilon,1}(s,x))\right] \ud W_x^\epsilon(s)\\
&+\frac{1}{2}\int_0^t\Psi_n''(v_\epsilon(s,x)) \left[\rho(u_{\epsilon,2}(s,x))-\rho(u_{\epsilon,1}(s,x))\right]^2\frac{1}{\sqrt{4\pi\epsilon}}\ud s\\
&+ \int_0^t \Psi_n'(v_\epsilon(s,x)) \tlMr{\epsilon}{\delta}{D}{a}  v_{\epsilon}(s,x) \ud s.
\end{align*}
% The condition $\rho(0)=0$ implies that $\rho(u)\le \LIP_\rho |u|$. Thus,
By the Lipschitz condition on $\rho$,
\[
\Psi_n''(v_\epsilon(s,x)) \left[\rho(u_{\epsilon,2}(s,x))-\rho(u_{\epsilon,1}(s,x))\right]^2
\le \LIP_\rho^2 \Psi_n''(v_\epsilon(s,x)) v_\epsilon^2(s,x)\le
2\LIP_\rho^2/n.
\]
Hence,
\begin{align*}
 \E\left[\Psi_n(v_\epsilon(t,x))\right]\le &\frac{\LIP_\rho^2 t}{n \sqrt{4\pi\epsilon}}
+ \E\left[\frac{1}{\epsilon}\int_0^t  \ud s \: \Psi_n'(v_\epsilon(s,x)) \int_\R \ud y \lMr{\delta}{G}{a}(\epsilon,x-y)\left[v_\epsilon(s,y)-v_\epsilon(s,x)\right]\right].
\end{align*}
Now let $n$ go to $+\infty$, by \eqref{E:PsiConv} and the monotone convergence theorem,
\begin{align*}
 \E\left[\Psi(v_\epsilon(t,x))\right]\le & \frac{1}{\epsilon}\int_0^t  \E\left[\one(v_\epsilon(s,x)<0)v_\epsilon(s,x)\right] \ud s\\
&-\frac{1}{\epsilon}\int_0^t\ud s\int_\R \ud y
\lMr{\delta}{G}{a}(\epsilon,x-y)
\E\left[\one(v_\epsilon(s,x)<0)v_\epsilon(s,y)\right].
\end{align*}
Notice that
\begin{align*}
 -\frac{1}{\epsilon}\int_0^t  \ud s\int_\R\ud y   &\lMr{\delta}{G}{a}(\epsilon,x-y) \E\left[\one(v_\epsilon(s,x)<0) v_\epsilon(s,y)\right]\\
\le&
 -\frac{1}{\epsilon}\int_0^t  \ud s\int_\R\ud y  \lMr{\delta}{G}{a}(\epsilon,x-y) \E\left[\one(v_\epsilon(s,x)<0,v_\epsilon(s,y)<0) v_\epsilon(s,y)\right]\\
=&\frac{1}{\epsilon}\int_0^t  \ud s\int_\R\ud y \lMr{\delta}{G}{a}(\epsilon,x-y) \E\left[\one(v_\epsilon(s,x)<0,v_\epsilon(s,y)<0) |v_\epsilon(s,y)|\right]\\
\le &\frac{1}{\epsilon}\int_0^t  \ud s\int_\R\ud y \lMr{\delta}{G}{a}(\epsilon,x-y) \E\left[\one(v_\epsilon(s,y)<0) |v_\epsilon(s,y)|\right]
\end{align*}
Then using the fact that $|x|\one(x<0) = \Psi(x)$, we have that
\[
 \E\left[\Psi(v_\epsilon(t,x))\right]\le \frac{1}{\epsilon}\int_0^t  \ud s\int_\R\ud
y
\lMr{\delta}{G}{a}(\epsilon,x-y) \E\left[\Psi(v_\epsilon(s,y))\right].
\]
Therefore, by Gronwall's lemma applied to $\sup_{y\in\R}\E\left[\Psi(v_\epsilon(s,y))\right]$, one can conclude that
$\E\left[\Psi(v_\epsilon(s,y))\right]=0$ for every $t>0$ and $x\in\R$. This proves \eqref{E:NonUe}.

{\vspace{0.7em}\noindent \bf Step 2.} In this step, we assume that
$\mu(\ud x) = f(x)\ud x$ with $f\in L^\infty(\R)$ and $f(x)\ge 0$ for all $x\in\R$.
Denote $f_\epsilon(x):= (\mu*\lMr{\delta}{G}{a}(\epsilon,\cdot))(x)$.
We will prove that
\begin{align}\label{E:uAprox}
\lim_{\epsilon\rightarrow 0}  \sup_{x\in\R}
\Norm{u_\epsilon(t,x)-u(t,x)}_2^2 = 0,\quad \text{for all $t>0$},
\end{align} where $u_\epsilon$ is a solution to \eqref{E:Ueps} with $u_\epsilon(0,x)=f_\epsilon(x)$ and $u$ is a solution to \eqref{E:FracHt}.

Fix $T>0$.
Notice that $u_\epsilon(t,x)$ can be written in the following mild form using the kernel of $\tlMr{\epsilon}{\delta}{\bold G}{a}(t)$ in \eqref{E:Geps}:
\begin{align*}
u_\epsilon(t,x)=&\left(f_\epsilon*\tlMr{\epsilon}{\delta}{G}{a}(t,\cdot)\right)(x)
+\int_0^t e^{-(t-s)/\epsilon}\rho\left(u_\epsilon(s,x)\right)\ud W_x^\epsilon(s)\\
&+\int_0^t\int_\R\tlMr{\epsilon}{\delta}{R}{a}(t-s,x-y) \rho\left(u_\epsilon(s,y)\right)\ud W_y^\epsilon(s)\ud y,
\end{align*}
where the last term equals to
\[
\int_0^t \int_\R \left(\int_\R \ud z \tlMr{\epsilon}{\delta}{R}{a}(t-s,x-z)\rho(u_\epsilon(s,z)) \phi_\epsilon(y-z)\right) W(\ud s,\ud y).
\]
By \eqref{E:RGaprx2} below, the boundedness of the initial data implies that for all $t>0$,
\begin{gather}\label{E:AT}
A_t:=\sup_{\epsilon\in\:]0,1]}\sup_{s\in[0,t]} \sup_{x\in\R} \Norm{u_\epsilon(s,x)}_2^2\vee \Norm{u(s,x)}_2^2<+\infty.
\end{gather}
%where $u(t,x)$ is the solution to \eqref{E:FracHt}; see \eqref{E:WalshSI}.
By the linear growth condition \eqref{E:LinGrow},

\begin{align*}
&\Norm{u_\epsilon(t,x)-u(t,x)}_2^2\\
&\hspace{1em}\le
6\left(f_{\epsilon}*\tlMr{\epsilon}{\delta}{G}{a}(t,\cdot)-f*\lMr{\delta}{G}{a}(t,\cdot)\right)^2(x) \\
&\hspace{2em} +6\Lip_\rho^2 \int_0^t \ud s  \: \frac{1}{\sqrt{4\pi\epsilon}}e^{-2(t-s)/\epsilon}\left(\Vip^2+\Norm{u_\epsilon(s,x)}_2^2\right)\\
&\hspace{2em}+6\int_0^t\ud s\int_\R\ud y\: \E\left[\left|
\int_\R \ud z \tlMr{\epsilon}{\delta}{R}{a}(t-s,x-z)
\left[\rho(u_\epsilon(s,z))-\rho(u(s,z))\right]\phi_\epsilon(y-z)
\right|^2\right]\\
&\hspace{2em}+6\int_0^t\ud s\int_\R\ud y\: \E\left[\left|
\int_\R \ud z \tlMr{\epsilon}{\delta}{R}{a}(t-s,x-z)
\left[\rho(u(s,z))-\rho(u(s,y))\right]\phi_\epsilon(y-z)
\right|^2\right]\\
&\hspace{2em}+6\int_0^t\ud s\int_\R\ud y\: \E\left[\left|
\int_\R \ud z
\left[\tlMr{\epsilon}{\delta}{R}{a}(t-s,x-z)-\lMr{\delta}{G}{a}(t-s,x-z)\right]  \rho(u(s,y))\phi_\epsilon(y-z)
\right|^2\right]\\
&\hspace{2em}+6\int_0^t\ud s\int_\R\ud y\: \E\left[\left|
\int_\R \ud z
\left[\lMr{\delta}{G}{a}(t-s,x-z)-\lMr{\delta}{G}{a}(t-s,x-y)\right]  \rho(u(s,y))\phi_\epsilon(y-z)
\right|^2\right]\\
&\hspace{1em}=: 6 \sum_{n=1}^6 I_n(t,x;\epsilon).
\end{align*}

Denote $C_f:=\sup_{x\in\R} f(x)\ge \sup_{x\in\R} f_\epsilon(x)$.
Using the semigroup property, we see that
\begin{align*}
& I_1(t,x;\epsilon)  \\
\le& \left[\left(f_\epsilon*\left|\tlMr{\epsilon}{\delta}{G}{a}(t,\cdot)-\lMr{\delta}{G}{a}(t,\cdot)\right|\right)(x) +\left(f*|\lMr{\delta}{G}{a}(t+\epsilon,\cdot)-\lMr{\delta}{G}{a}(t,\cdot)|\right)
(x) \right]\\
& \times \left[\left(f_\epsilon*\tlMr{\epsilon}{\delta}{G}{a}(t,\cdot)\right)(x)+\left(f*\lMr{\delta}{G}{a}(t,\cdot)\right)(x)\right]\\
\le &
2\: C_f^2\left(e^{-t/\epsilon}+\int_\R \ud y \left|\tlMr{\epsilon}{\delta}{R}{a}(t,y)-\lMr{\delta}{G}{a}(t,y)\right|+\int_{\R}\ud y \left|
\lMr{\delta}{G}{a}{(t+\epsilon,y)}-\lMr{\delta}{G}{a}(t,y) \right| \right)\\
\le &
2\: C_f^2
\left(2e^{-t/\epsilon}+(C'+C'') \left(\epsilon/t\right)^{1/2} \right),
\end{align*}
where the last step is due to Lemma \ref{L:RGaprxL1}, \eqref{E:timediff} and the fact that $\log(1+x) \leq \sqrt{x}$ for all $x\geq 0$, and  $C'$ and $C''$ are the constants defined in
\eqref{E:RGaprxL1} and \eqref{E:constantc'}. For simplicity, define $C_*:=C'+C''$.

As for $I_2(t,x;\epsilon)$,
\[
I_2(t,x;\epsilon)\le \Lip_\rho^2\: (\Vip^2+A_t)\frac{\sqrt{\epsilon}}{\sqrt{4\pi}}
 \left(1- e^{-2t/\epsilon }\right),
\]
which implies
\[
\lim_{\epsilon\rightarrow 0}\sup_{0\le t\le T}\sup_{x\in\R}I_2(t,x;\epsilon) =0.
\]
By the H\"older inequality and the Lipschitz continuity of the function $\rho$,
\begin{align}\notag
I_3(t,x;\epsilon)&
\le \LIP_\rho^2 \int_0^t\ud s\int_\R\ud y\int_\R\ud z \tlMr{\epsilon}{\delta}{R}{a}^2(t-s,x-z)\Norm{u_\epsilon(s,z)-u(s,z)}_2^2 \phi_\epsilon(y-z)\\
&\le \LIP_\rho^2 \int_0^t\ud s\int_\R\ud y\tlMr{\epsilon}{\delta}{R}{a}^2(s,x-y)\Norm{u_\epsilon(t-s,y)-u(t-s,y)}_2^2,
\label{E:I3}
\end{align}
and similarly,
\begin{align*}
I_4(t,x;\epsilon)
& \le \LIP_\rho^2 \int_0^t\ud s\int_\R\ud y\int_\R\ud z \tlMr{\epsilon}{\delta}{R}{a}^2(s,z)\Norm{u(t-s,x-z)-u(t-s,x-y)}_2^2 \phi_\epsilon(y-z).
\end{align*}
By Lemma \ref{L:MInc}, for some constant $C:=C(T,a,\delta,\mu)$,
\[
\Norm{u(t-s,x-z)-u(t-s,x-y)}_2^2 \le C (t-s)^{-1/a}|y-z| + C A_T |y-z|^{a-1}.
\]
Hence, integrating over $\ud y$ first and then integrating over $\ud z$ using \eqref{E:RGaprx2} give that
\begin{align*}
I_4(t,x;\epsilon)
 &\le \LIP_\rho^2 C \int_0^t\ud s\int_\R\ud z \tlMr{\epsilon}{\delta}{R}{a}^2(s,z)
\left[(t-s)^{-1/a} \sqrt{2\epsilon/\pi}+\frac{A_T 2^{(a-1)/2}}{\sqrt{\pi}}\Gamma\left(a/2\right)\epsilon^{(a-1)/2}
\right]\\
&\le\frac{\LIP_\rho^2 C \: C_{a,\delta}}{\sqrt{\pi}} \int_0^t \ud s
\: s^{-1/a}\left[(t-s)^{-1/a} \sqrt{2\epsilon}+A_T 2^{(a-1)/2}\Gamma\left(a/2\right)\epsilon^{(a-1)/2}
\right],
\end{align*}
where $C_{a,\delta}$ is defined in Lemma \ref{L:RGaprx}. Finally, integrating over $\ud s$ using the Beta integral, we have that for some finite constant $C^*:=C^*(T,a,\delta,\mu,A_T)>0$,
\[
I_4(t,x;\epsilon)\le C^*  \LIP_\rho^2 \left(t^{1-2/a} \epsilon^{1/2}+\epsilon^{(a-1)/2}\right).
\]
By H\"older inequality, \eqref{E:LinGrow} and \eqref{E:AT},
\[
I_5(t,x;\epsilon)\le \Lip_\rho^2 \left(\Vip^2+A_t\right)\int_0^t\ud s\int_\R\ud y \int_\R\ud z \: \phi_\epsilon(y-z)
\left[\tlMr{\epsilon}{\delta}{R}{a}(t-s,x-z)-\lMr{\delta}{G}{a}(t-s,x-z)\right]^2.
\]
Integrate $\ud y$ and enlarge the integral interval for $\ud s$ from $[0,t]$ to $[0,T]$,
\[
I_5(t,x;\epsilon)\le \Lip_\rho^2 \left(\Vip^2+A_T\right)\int_0^T\ud s\int_\R \ud z \:
\left[\tlMr{\epsilon}{\delta}{R}{a}(s,z)-\lMr{\delta}{G}{a}(s,z)\right]^2,
\]
and then apply \eqref{E:RGaprx1} to obtain
\[
\lim_{\epsilon\rightarrow 0} \sup_{0\le t\le T}\sup_{x\in\R} I_5(t,x;\epsilon) =0.
\]
Similarly to the case of $I_5$, we have that
\begin{align*}
I_6(t,x;\epsilon)\le & \Lip_\rho^2 \left(\Vip^2+A_T\right)\int_0^T\ud s\int_\R\ud y \int_\R\ud z \: \phi_\epsilon(y-z)
\left[\lMr{\delta}{G}{a}(s,x-z)-\lMr{\delta}{G}{a}(s,x-y)\right]^2\\
= &
 2\Lip_\rho^2 \left(\Vip^2+A_T\right)\int_0^T\ud s\int_\R\ud y\left[
\lMr{\delta}{G}{a}^2(s,y)
-\lMr{\delta}{G}{a}(s,y) \int_\R\ud z \lMr{\delta}{G}{a}(s,z) \phi_\epsilon(y-z)\right].
\end{align*}
Define $F_\epsilon(s,y):=\lMr{\delta}{G}{a}^2(s,y)
-\lMr{\delta}{G}{a}(s,y) \int_\R\ud z \lMr{\delta}{G}{a}(s,z) \phi_\epsilon(y-z)$. Clearly, $\lim_{\epsilon\rightarrow 0}F_\epsilon(s,y)=0$ for all $s>0$ and $y\in\R$.
On the other hand,
\[
F_\epsilon(s,y)\le
\lMr{\delta}{G}{a}^2(s,y) + \lMr{\delta}{G}{a}(s,y)s^{-1/a} \Lambda,
\]
where the constant $\Lambda$ is defined in \eqref{E:Cst-dLa}. In fact, this upper bound is integrable:
\begin{align*}
\int_0^T\ud s\int_\R\ud y &\left(\lMr{\delta}{G}{a}^2(s,y) + \lMr{\delta}{G}{a}(s,y)s^{-1/a} \Lambda\right)\\
&=\int_0^T \ud s \left[\lMr{\delta}{G}{a}(2s,0)+ \Lambda s^{-1/a}\right]
= \left(\frac{\lMr{\delta}{G}{a}(1,0)}{2^{1/a}}+\Lambda\right) \frac{a}{a-1}\:T^{1-1/a}.
\end{align*}
Hence, the dominated convergence theorem implies that
\[
\lim_{\epsilon\rightarrow 0}\sup_{0\le t\le T}\sup_{x\in\R}I_6(t,x;\epsilon) =0.
\]
Now set $M(t;\epsilon):=\sup_{y\in\R} \Norm{u_\epsilon(t,y)-u(t,y)}_2^2$.
Fix $T>0$. Combining things together,  we can get that for some constant $C_T>0$,
\[
M(t;\epsilon)\le C_T\int_0^t\ud s\: (t-s)^{-1/a} M(s;\epsilon)
+ H(T;\epsilon) + \widehat{H}(t;\epsilon),
\]
where
\begin{align*}
H(T;\epsilon) & :=6\sum_{n=2,5,6}\sup_{0\le t\le T}\sup_{x\in\R} I_n(t,x;\epsilon),\\
 \widehat{H}(t;\epsilon) &:=12\: C_f^2
\left(2e^{-t/\epsilon}+C_* \left(\epsilon/t\right)^{1/2}\right) + 6\: C^*  \LIP_\rho^2 \left(t^{1-2/a} \epsilon^{1/2}+\epsilon^{(a-1)/2}\right).
\end{align*}
Then by Chandirov's lemma, which is a variation of Bellman's inequality (see \cite[Theorem 1.4, on p. 5]{BainovSimeonov92}), for $0<t\le T$,
\begin{align*}
 M(t;\epsilon) &\le H(T;\epsilon)+ \widehat{H}(t;\epsilon)
+\int_0^t  \ud s \left(H(T;\epsilon)+ \widehat{H}(s;\epsilon)\right)(t-s)^{-1/a} \,
\exp\left(\int_s^t \ud \tau \: (t-\tau)^{-1/a}\right)\\
&= H(T;\epsilon)+ \widehat{H}(t;\epsilon)
+\int_0^t  \ud s \left(H(T;\epsilon)+ \widehat{H}(s;\epsilon)\right)(t-s)^{-1/a}\,
\exp\left(\frac{a}{a-1}(t-s)^{1-1/a}\right)\\
&\le H(T;\epsilon)\left(1+\frac{a}{a-1} T^{1-1/a}\,\exp\left(\frac{a\:T^{1-1/a}}{a-1}\right)\right)\\
&\qquad\qquad\qquad+ \widehat{H}(t;\epsilon)+\int_0^t  \ud s\:  \widehat{H}(s;\epsilon)(t-s)^{-1/a}\,
\exp\left(\frac{a\: (t-s)^{1-1/a}}{a-1}\right).
\end{align*}
Clearly, as $\epsilon \rightarrow 0$, the first two terms in the above upper bound go to zero.
The integral also goes to zero by applying the dominated convergence theorem.
This proves \eqref{E:uAprox}.

Finally, suppose that $\mu_i(\ud x)=f_i(x)\ud x$ with $f_i\in L^\infty(\R)$, $i=1,2$.
If $f_1(x)\le f_2(x)$ for almost all $x\in\R$, then by Step 1 we know that $v_\epsilon(t,x):= u_{\epsilon,2}(t,x)- u_{\epsilon,1}(t,x)\ge 0$ for all $t>0$ and $x\in\R$, a.s. Then Step 2 implies
$v_\epsilon(t,x)$ converges to $v(t,x)=u_2(t,x) - u_1(t,x)$ in $L^2(\Omega)$ for all $t>0$ and $x\in\R$. Therefore,  the nonnegativity of $v(t,x)$ is inherited from that of $v_\epsilon(t,x)$, that is,
\[
P(u_1(t,x)\le u_2(t,x),\:\text{for all $t>0$ and $x\in\R$}) =1.
\]

{\vspace{0.7em}\noindent \bf Step 3.}
Now we assume that $\mu_i \in\calM_a^*(\R)$.
Recall the definition of $\psi_\epsilon$ in \eqref{E:psi}.
Fix $\epsilon>0$. Let $u_{\epsilon,i}$, $i=1,2$, be the solutions to \eqref{E:FracHt} starting from
$\left([\mu_i \psi_\epsilon]*\lMr{\delta}{G}{a}(\epsilon,\cdot)\right)(x)$.
Denote $v(t,x)=u_2(t,x)-u_1(t,x)$ and $v_\epsilon(t,x)=u_{\epsilon,2}(t,x)-u_{\epsilon,1}(t,x)$.
Because $\psi_\epsilon$ is a continuous function with compact support on $\R$, the initial data for $u_{\epsilon,i}(t,x)$ is bounded:
\[
\sup_{x\in\R} |\left([\mu_i \psi_\epsilon]*\lMr{\delta}{G}{a}(\epsilon,\cdot)\right)(x)|
\le
\frac{\Lambda}{\epsilon^{1/a}} \int_\R \psi_\epsilon(y)|\mu_i|(\ud y)<+\infty,
\]
where $\Lambda$ is defined in \eqref{E:Cst-dLa}. Hence, by Step 2, we have that
\[
P(v_\epsilon(t,x)\ge 0,\:\text{for all $t>0$ and $x\in\R$}) =1,\quad\text{for all  $\epsilon>0$.}
\]
Applying Theorem \ref{T:Approx}, we obtain
\[
P\left(v(t,x)\ge 0,\:\text{for all $t>0$ and $x\in\R$}\right) =1.
\]
This completes the proof of Theorem \ref{T:WComp}.
\end{proof}

% \begin{remark}
% The mollifier $\phi_\epsilon(x)$ in the proof of Theorem \ref{T:WComp} cannot be a smooth function with compact support. Let $\phi(x)$ be a smooth function with compact support and $\phi_\epsilon(x)
% =\epsilon^{-1}\phi(x/\epsilon)$. Then \eqref{E:Quad} becomes
% \[
% \ud \InPrd{W_x^\epsilon(t)} = \frac{C}{\epsilon} \: \ud s, \quad \text{where $C=\int_{\R}\phi^2(x)\ud x$.}
% \]
% Then the integration in $I_2(t,x)$ does not converge.
% \end{remark}

\section{Proof of Theorem \ref{T:SComp}} \label{S:Scomp}

We need several lemmas.
Lemma \ref{L:J0LInd} below plays a role to initialize the induction procedure.

\begin{lemma}\label{L:J0LInd}
Let $d>0$. For all $t>0$ and $M>0$, there exist some constants $1<m_0=m_0(t,M)<\infty$ and
$0<\gamma \leq 1/4$ such that for all $m\ge m_0$, all $s\in \left[t/(2m),t/m\right]$ and $x\in\R$,
\begin{align}\label{E:J0LInd1}
\left(\lMr{\delta}{G}{a}(s,\cdot)*1_{]-d,d[}(\cdot)\right)(x)
\ge \gamma\, 1_{]-d-M/m,d+M/m[}(x).
\end{align}
\end{lemma}
\begin{proof}
Let $Z$ be a random variable with the stable density $\lMr{\delta}{G}{a}(1,x)$.
Define $\gamma:=\min\{P(Z\le 0), P(Z\ge 0)\}/2$. Clearly, $0<\gamma\le 1/4$.
We first consider the case where $-d-M/m\leq x \le 0$.
Because  $t/2  \le ms\le t$, we have
\begin{align*}
\left(\lMr{\delta}{G}{a}(s,\cdot)*1_{]-d,d[}(\cdot)\right)(x) &=
\int_{-d}^d \lMr{\delta}{G}{a}(s,x-y)\ud y= \int^{\frac{x+d}{s^{1/a}}}_{\frac{x-d}{s^{1/a}}} \lMr{\delta}{G}{a}(1,z)\ud z
\\
& \ge P\left(-d\, (2m)^{1/a} t^{-1/a} \le Z\le -M\,m^{(1-a)/a} t^{-1/a} \right).
\end{align*}
Similarly, when $0 \le x \le d+M/m$, we have
\begin{align*}
\left(\lMr{\delta}{G}{a}(s,\cdot)*1_{]-d,d[}(\cdot)\right)(x)
&\ge P\left(M\, t^{-1/a} m^{(1-a)/a} \le Z \le d (2m)^{1/a} t^{-1/a}\right).
\end{align*}
Therefore, when $m$ is large enough, the above probabilities are bigger than $\gamma$. 
This completes the proof of Lemma \ref{L:J0LInd}.
\end{proof}

\begin{lemma}\label{L:MBds}
(1) If $\mu\in\calM_a(\R)$ and $\rho$ satisfies \eqref{E:LinGrow}, then for all $p\ge 2$,
there exists some finite constant $C:=C(a,\delta,\Lip_\rho,\Vip,\mu,p)>0$ such that,
\[
 \sup_{x\in\R}
\Norm{u(t,x)}_p^2 \le C (t\vee 1)^{2(1+1/a)}t^{-2/a}\left[1+t^{1-1/a}+ t\exp\left( \gamma^{a^*} t\right)\right],
\]
for all $t>0$, where $\gamma:=8p\Lip_\rho\Lambda\Gamma(1/a^*)$.\\
(2) If $\mu(\ud x) = c\,\ud x$, $c\ne 0$ and $\rho(0)=0$, then for some constant $Q:=Q(c,a,\LIP_\rho,\Lambda)>0$,
\[
 \sup_{x\in\R}\E\left(|u(t,x)|^p\right)\le Q^p \exp\left( Q p^{\frac{2a-1}{a-1} } t\right),\quad
\text{for all $p\ge 2$ and $t\ge 0$}.
\]
\end{lemma}
\begin{proof}
Part (1) is from \cite[Lemma 4.9 and (4.20)]{ChenDalang14FracHeat}.
As for part (2), notice that $J_0(t,x) \equiv c$. Then by \eqref{E:MomUp} and \eqref{E:UpBd-K}, for $p\ge 2$ and $p\in\bbN$,
\begin{align*}
\Norm{u(t,x)}_p^2 & \le 2c^2 +C \:c^2\int_0^t\ud s\:  \left(s^{-1/a}+\exp\left(\gamma_p^{a^*} s\right)\right)\\
&\le
2c^2 +C \:c^2\left(\frac{a}{a-1} t^{\frac{a-1}{a}} + \gamma_p^{-a^*} \exp\left(\gamma_p^{a^*} t\right)\right),
\end{align*}
where $\gamma_p=16\: p \Lip_\rho^2 \Lambda \Gamma(1/a^*)$ and the constant $C=C(\LIP_\rho)$ is defined in Proposition \ref{P:UpperBdd-K}. Notice that $\log(x)\le \beta x$ for all $x\ge 0$
whenever $\beta\ge e^{-1}$.
So by choosing $\theta= \frac{a-1}{a e}\gamma_2^{-a^*}$, we have that $\exp\left(\theta \gamma_p^{a^*} t\right) \ge  t^{\frac{a}{a-1}}$ for all $t>0$.
Hence, if $c\ne 0$, then
\[
\Norm{u(t,x)}_p^2 \le \left[ 2c^2 +
C \: c^2\left(\frac{a}{a-1} + \gamma_2^{-a^*} \right)\right]\exp\left(\theta \gamma_p^{a^*} t\right).
\]
Then, raise both sides by a power of $p/2$. This completes the proof of Lemma \ref{L:MBds}.
\end{proof}

% Recall that $I(t,x)$ denotes the stochastic integral in \eqref{E:WalshSI}.
The following lemma proves the inductive step.

\begin{lemma}\label{L:LgDiv}
Let $d>0, t>0$ and $M>0$. If $\rho(0)=0$ and $\mu(\ud x)=1_{[-d,d]}(x)\ud x$, then there are some finite constants $Q:=Q(\beta, \LIP_\rho, \Lambda, t)>0$, $0<\beta\le 1/8$, and $m_0>0$ such that
for
all $m\ge m_0$,
\begin{align*}
P\Big(
u(s,x)\ge \beta 1_{]-d-M/m,d+M/m[}(x)\;\;&\text{for all $\frac{t}{2m}\le s \le \frac{t}{m}$ and $x\in\R$}
\Big)\\
&\ge 1 -\exp\left(- Q \: m^{1-1/a} [\log(m)]^{2-1/a}\right).
\end{align*}
\end{lemma}

\begin{proof}
Define $S:=S_{t,m,d,M}:=\{(s,y): t/(2m)\le s\le t/m, \; |x|\le d+M/m \}$.
By Lemma \ref{L:J0LInd}, for some constant $0<\beta\le 1/8$,
\begin{align}\label{E_:2beta}
\left(\mu * \lMr{\delta}{G}{a}(s,\cdot)\right)(x)
\ge 2\beta 1_{]-d-M/m,d+M/m[}(x)\quad \text{for all $s\in \left[t/(2m),t/m\right]$ and $x\in\R$.}
\end{align}
Hence,
\begin{align*}
 P\Bigg(&
u(s,x) < \beta 1_{]-d-M/m,d+M/m[}(x)\;\;\text{for some $\frac{t}{2m}\le s \le \frac{t}{m}$ and $x\in\R$}
\Bigg)\\
 & \le
 P\Bigg(
I(s,x) < - \beta\;\;\text{for some $(s,x)\in S$}
\Bigg) \\
&\le
 P\left(
\sup_{(s,x)\in S}|I(s,x)| >\beta
\right)\le
\beta^{-p}
\E\left[
\sup_{(s,x)\in S}|I(s,x)|^p\right],
\end{align*}
where we have applied Chebyshev's inequality in the last step.
Denote $\tau=t/m$ and $S':=\{(s,y): 0\le s\le t/m, \; |x|\le d+M/m \}$.
By the fact that $I(0,x)\equiv 0$ for all $x\in\R$,  a.s., we see that for all $0<\eta< 1-\frac{2(a+1)}{p(a-1)}$,
\begin{align}
\E\left[
\sup_{(s,x)\in S}\left|\frac{I(s,x)}{\tau^{\frac{a-1}{2a}\eta}}\right|^p\right]&\le
\E\left[
\sup_{(s,x), (s',x')\in S'}\left|\frac{I(s,x)-I(s',x')}{\left(|x-x'|^{\frac{a-1}{2}}+|s-s'|^{\frac{a-1}{2a}}\right)^\eta}\right|^p\right].
\label{E_:Holder}
\end{align}

Let us find the upper bound of \eqref{E_:Holder}.  By the Burkholder-Davis-Gundy inequality  and \cite[Proposition 4.4]{ChenDalang14FracHeat}, for some universal constant $C_1>0$,
\[
\E\left[|I(s,x)-I(s',x')|^p\right]\le C_1  \left(|x-x'|^{\frac{a-1}{2}}+|s-s'|^{\frac{a-1}{2a}}\right)^{p/2} \sup_{(t,y)\in S'} \|u(t,y)\|_p^p\:,
\]
for all $(s,x)$ and $(s',x')\in S'$. Hence, by part (1) of Lemma \ref{L:MBds}, 
\[
\sup_{(t,y)\in S'} \|u(t,y)\|_p^p \le Q^p \exp\left(Q p^{\frac{2a-1}{a-1}} \tau\right)=:C_{p,\tau},
\]
for some constant $Q:=Q(a,\LIP_\rho,\Lambda)>0$ and for all $p\in [2,\infty[\:$.
Notice that when $m$ is sufficiently large, $S\subset [0,1]^2$. Hence, the right hand side of \eqref{E_:Holder} is bounded by the same quantity with $S$ replaced by $[0,1]^2$.
Then by Kolmogorov's continuity theorem (see \cite[Theorem 1.4.1]{Kunita90Flow} and \cite[Proposition 4.2]{ChenDalang13Holder}), for some universal constant $C>0$, the
expectation on the right hand side of \eqref{E_:Holder} is bounded above by $C^p C_{p,\tau}$.

We consider the case where $p=O\left([m\log m]^{1-1/a}\right)$ as $m\rightarrow \infty$ (see \eqref{E:pA} below).
In this case, we have $p^{a/(a-1)}\tau=O(\log m)$ as $m\rightarrow\infty$ since $\tau=t/m$.
This implies that there exists some constant $Q':=Q'(\beta, \LIP_\rho, \Lambda, t)$ such that
\begin{align*}
\beta^{-p}\E\left[
\sup_{(s,x)\in S}\left|I(s,x)\right|^p\right]
&\le
Q'\, \tau^{\frac{(a-1)\eta}{2a}p}\, \exp\left(Q' p^{\frac{2a-1}{a-1}} \tau\right)\\
&=
Q' \exp\left(Q' p^{\frac{2a-1}{a-1}} \tau+\frac{(a-1)\eta}{2a}p \log(\tau)\right).
\end{align*}
By denoting $\eta = \theta \left(1-\frac{2}{p}\:
\frac{a+1}{a-1}\right)$ with $\theta \in \:]0,1[\;$, the above exponent becomes
\[
f(p):=Q' \tau  p^{\frac{2a-1}{a-1}}+\frac{\log (\tau )\:\theta\: \left(p[a-1]-2 [a+1]\right)}{2 a}.
\]
It is easy to see that  $f(p)$ for $p\ge 2$ is minimized at
\begin{align*}
p= & \left(\frac{(a-1)^2 \theta  \log (1/\tau )}{2 a
   (2 a-1) Q' \tau }\right)^{1-1/a}=
\left(\frac{(a-1)^2 \theta \: m \log(m/t)}{2 a (2 a-1) Q' t }\right)^{1-1/a}.
\end{align*}
Thus, for some constants $A:=A(\beta, \LIP_\rho, \Lambda, t)$ and $Q'':=Q''(\beta, \LIP_\rho, \Lambda, t)$,
\begin{align}\label{E:pA}
\min_{p\ge 2} f(p) \le f(p') = - Q'' m^{1-1/a} [\log(m)]^{2-1/a},\quad\text{with $p'= A  \left[m \log(m)\right]^{1-1/a}.$}
\end{align}
This completes the proof of Lemma \ref{L:LgDiv}.
\end{proof}

\begin{figure}[htbp]
\centering
 \includegraphics{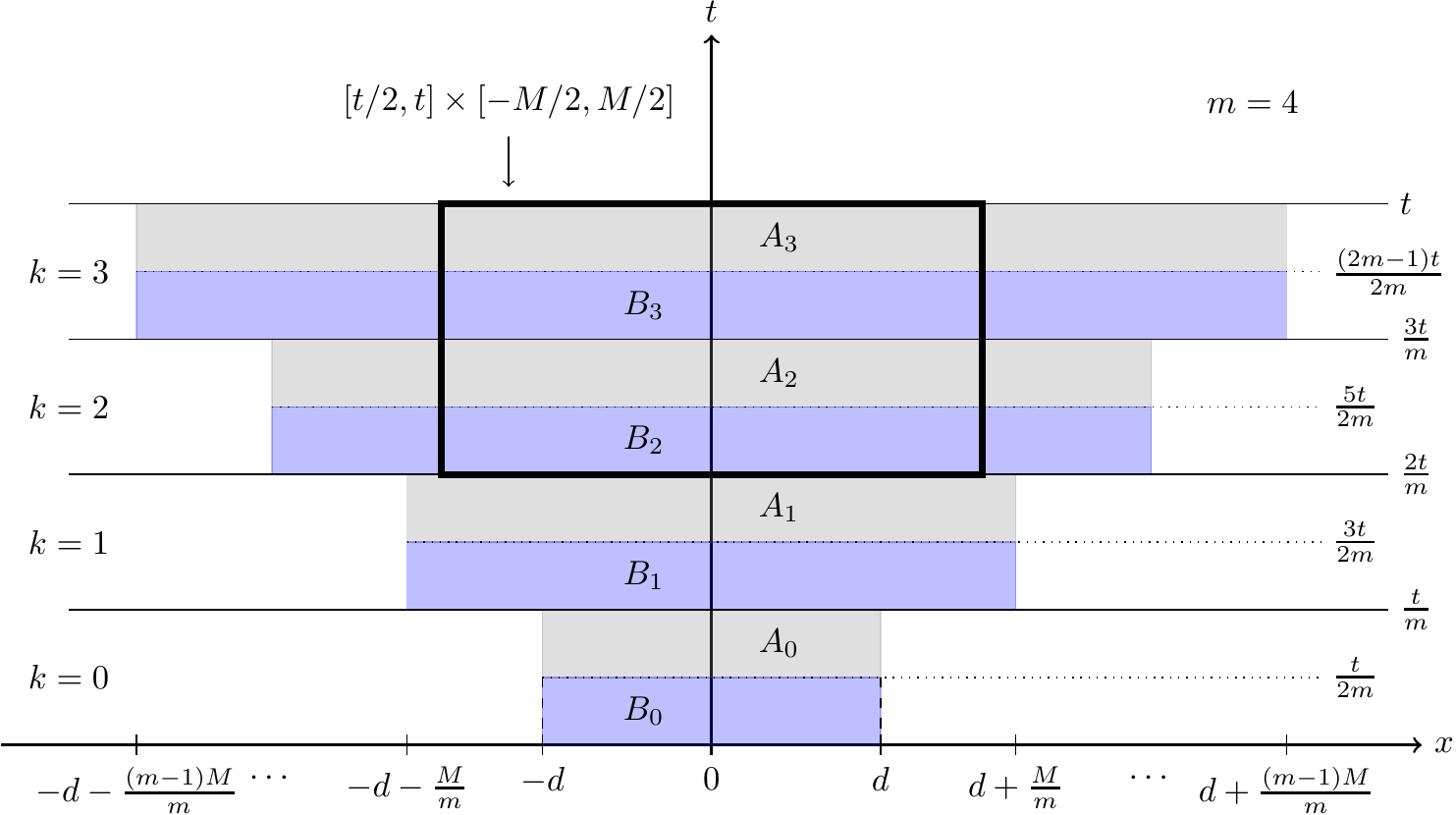}
\caption{Induction schema for the strong comparison principle.}
\label{F:SComp}
\end{figure}

\begin{proof}[Proof of Theorem \ref{T:SComp}]
Let $u(t,x):=u_2(t,x)-u_1(t,x)$ and $\tilde \rho(u):=\rho(u+u_1)-\rho(u_1)$.  Then $u(t,x)$ is, in fact, a solution to \eqref{E:FracHt} with the nonlinear function $\tilde\rho$ and the
initial data $\mu:=\mu_2-\mu_1$. We note that $\tilde\rho(0)=0$ and $\tilde\rho$ is a Lipschitz continuous function with the same Lipschitz constant as for $\rho$. For simplicity, we will use $\rho$
instead of $\tilde\rho$.  By the weak comparison principle (Theorem \ref{T:WComp}), we only need to consider the case when $\mu$ has compact support and prove that $u(t,x)>0$ for all $t>0$ and
$x\in\R$, a.s.

% 
% Recall that by Lebesgue's decomposition theorem (see, e.g., \cite[Theorem 19.61]{HewittStromberg75}),
% any (regular) Borel measure $\mu$ can be expressed in a unique way in the form
% \[
% \mu=\mu_a+\mu_s+\mu_d,
% \]
% where $\mu_a$ is absolutely continuous (with respect to the Lebesgue measure), $\mu_s$ is singular continuous, and $\mu_d$ is purely discontinuous.

\paragraph{Case I.} We first assume that $\mu(\ud x)=f(x)\ud x$ with $f\in C(\R)$ and $f(x)\ge 0$ for all $x\in\R$.
Since $\mu>0$, there exists $x\in\R$ such that $f(x)>0$.
By the weak comparison principle (Theorem \ref{T:WComp}),
we only need to consider the case where $f(x)=1_{[-d,d]}(x)$ for some $d>0$ (this is because, if $f(x)=1_{[a,b]}(x)$, then we can use $f(x-(a+b)/2)$ as our initial function; in addition, if $f(x)=c1_{[-d,d]}(x)$, then we can consider $\tilde u(t,x):=c u(t,x)$ which is the unique solution to \eqref{E:FracHt} with the
initial function $1_{[-d,d]}(x)$ and with replacing $\rho(z)$ by $c\rho(z/c)$ which is also Lipschitz continuous with the same Lipschitz constant as $\rho(z)$).

Let $\gamma\in \:]0,1/4]$ be the constant defined in Lemma \ref{L:J0LInd} and let $\beta:=\gamma/2$.
For any $M>0$ and $k=0,1,\cdots,m-1$, define the events
\begin{align*}
A_k &:= \left\{u(s,x)\ge \beta^{k+1} 1_{S_k^m}(x)\:\text{for all $s\in\left[\frac{(2k+1)t}{2m},\frac{(k+1)t}{m} \right]$ and $x\in\R$}\right\},\\
B_k &:= \left\{u(s,x)\ge \beta^{k+1} 1_{S_k^m}(x)\:\text{for all $s\in\left[\frac{kt}{m},\frac{(2k+1)t}{2m} \right]$ and $x\in\R$}\right\}, 
% B_0& :=\left\{u\left(\frac{t}{2m},x\right)\ge \beta 1_{S_0^m}(x)\:\text{for all $x\in\R$}\right\},
\end{align*}
where 
\[
S^m_k :=\: \left ]-d-\frac{Mk}{m},d+\frac{Mk}{m}\right[\:.
\]
See Figure \ref{F:SComp} for an illustration of the schema.

By Lemma \ref{L:LgDiv}, there are constants $Q>0$ and $m_0>0$ such that for all $m\ge m_0$,
\[
P(A_0)\ge 1-c(m),
\]
where
\begin{equation}\label{E:CM}
c(m):=\exp\left(- Q \: m^{1-1/a} [\log(m)]^{2-1/a}\right).
\end{equation}
By definition, on the event $A_{k-1}$, $k\ge 1$,
\[
u\left(\frac{k t}{m},x\right)\ge  \beta^{k} 1_{S_{k-1}^m}(x),\quad\text{for all $x\in\R$.}
\]
Let $w_k(s,x)$ be the solution to the following SPDE:
\begin{align*}
 \begin{cases}
  \left(\displaystyle\frac{\partial}{\partial t} - \Dxa \right) w_k(s,x) =
\rho_k\left(w_k(s,x)\right) \W_k(s,x),& s\in \R_+^*:=\;]0,+\infty[\;,\: x\in\R,\cr
w_k(0,x) = 1_{S_{k-1}^m}(x),
 \end{cases}
\end{align*}
where $\rho_k(x):=\beta^{-k} \rho(\beta^{k} x)$ and $\{\W_k(s,x):=\W(s+kt/m,x)\}_{k\ge 1}$ is the time-shifted white noise.
Note that $\rho_k(x)$ is also a Lipschitz continuous  function with the
same Lipschitz constant as for $\rho$ and $\rho_k(0)=0$. Thus, by Lemma \ref{L:LgDiv}, we see that by the same constants $Q$ and $m_0$ as in \eqref{E:CM},  for all $m\ge m_0$,
\begin{align}\label{E:wk}
P\left( w_k(s,x) \geq \beta 1_{S_{k}^m}(x) \: \text{for all $s \in \left[\frac{t}{2m}, \frac{t}{m} \right]$} \:\text{and $x\in\R$}\right) \geq 1-c(m).
 \end{align}
Let $v(s,x)$ be a solution to \eqref{E:FracHt} with the initial data $\mu(\ud x):=\beta^k 1_{S_{k-1}^m}(x)\ud x$,
subject to the above time-shifted noise $\W_k$ with the same $\rho$. Then $v(s,x)=\beta^k w_k(s,x)$ a.s. for all $s\ge 0$ and $x\in\R$.
Since $u(s+kt/m,x) \ge v(s,x)$ for all $x\in\R$ and $s\ge 0$ by the Markov property and the weak comparison principle (Theorem \ref{T:WComp}), \eqref{E:wk} implies that
\[
P\left(A_k\mid \calF_{kt/m}\right)\ge 1- c(m), \quad\text{a.s. on $A_{k-1}$ for $1\le k\le m-1$.}
\]
Hence,
\[
P\left(A_k\mid A_{k-1}\cap \cdots \cap A_0\right) \ge 1-c(m),\quad\text{for all $1\le k\le m-1$.}
\]
Furthermore, because $A_0\subseteq B_0$, on the event $A_0$, we see that
\[
P(B_0) \ge P(A_0) \ge 1-c(m).
\]
Similarly, one can prove that
\[
P\left(B_k\mid B_{k-1}\cap \cdots \cap B_0\right) \ge 1-c(m),\quad\text{for all $1\le k\le m-1$.}
\]
Then,
\begin{align}\notag
 P\left(\cap_{0\le k\le m-1} \left[A_k \cap B_k\right] \right)
&\ge 1-\left(1-P\left(\cap_{0\le k\le m-1}A_k\right)\right)-\left(1-P\left(\cap_{0\le k\le m-1}B_k\right)\right)\\ \notag
&\ge (1-c(m))^{m-1}P(A_0)+ (1-c(m))^{m-1}P(B_0) -1\\
&\ge 2(1-c(m))^{m}-1.
\label{E_:abcup}
\end{align}
Therefore, for all $t>0$ and $M>0$,
\begin{align*}
 P\left(u(s,x)> 0 \;\; \text{for all $t/2\le s\le t$ and $|x|\le M/2$}\right)
&\ge
\lim_{m\rightarrow\infty} P\left(\cap_{0\le k\le m-1} \left[A_k \cap B_k\right] \right)
\\
&\ge \lim_{m\rightarrow\infty} 2(1-c(m))^{m}-1 =1.
\end{align*}
Since $t$ and $M$ are arbitrary, this completes the proof for Case I.

\paragraph{Case II.} Now we assume that $\mu\in\calM_{a,+}^*(\R)$.
We only need to prove that for each $\epsilon>0$,
\begin{align}\label{E:uEpsilon}
P\left(u(t,x)>0\; \text{for $t\ge \epsilon$ and $x\in\R$}\right) =1.
\end{align}
Fix $\epsilon>0$. Denote $V(t,x):=u(t+\epsilon,x)$. By the Markov property, $V(t,x)$ solves \eqref{E:FracHt} with
the time-shifted noise $\W_\epsilon (t,x):=\W(t+\epsilon,x)$ starting from $V(0,x)=u(\epsilon,x)$, i.e.,
\begin{equation}
 \label{E:VInt}
\begin{aligned}
 V(t,x) &= \left(u(\epsilon,\circ)* \lMr{\delta}{G}{a}(t,\circ)\right)(x)
+ \iint_{[0,t]\times\R} \rho(V(s,y)) \lMr{\delta}{G}{a}(t-s,x-y)W_\epsilon(\ud s,\ud y)\\
&=: \widetilde{J}_0(t,x) + \widetilde{I}(t,x).
\end{aligned}
\end{equation}

We first claim that
\begin{align}\label{E:uepsilonx}
P\left(u(\epsilon,x)= 0, \;\text{for all $x\in\R$}\right) =0.
\end{align}
Notice that by Theorem \ref{T:Holder}, the function $x\mapsto u(t,x)$ is H\"older continuous over $\R$
a.s. The weak comparison principle (Theorem \ref{T:WComp}) shows that $u(t,x)\ge 0$ a.s.
Hence, if \eqref{E:uepsilonx} is not true, then by the Markov property and the strong comparison principle in Case I, at all times $\eta\in [0, \epsilon]$, with some strict positive probability,
$u(\eta,x)=0$ for all $x\in\R$, which contradicts Theorem \ref{T:WeakSol} as $\eta$ goes to zero.
Therefore, there exists a sample space $\Omega'$ with $P(\Omega')=1$ such that for each $\omega\in\Omega'$,
there exists $x\in\R$ such that $u(\epsilon,x,\omega)>0$.

Since $u(\epsilon,x,\omega)$ is continuous at $x$, one can find two nonnegative constants $c$ and $\beta$ such that
$u(\epsilon,y,\omega)\ge \beta 1_{[x-c,x+c]}(y)$ for all $y\in\R$.
Then Case I implies that
\[
P\left(V_{\omega}(t,x)>0\;\text{for all $t\ge 0$ and $x\in\R$}\right) =1,
\]
where $V_\omega$ is the solution to \eqref{E:VInt} starting from $u(\epsilon,x,\omega)$. Therefore, \eqref{E:uEpsilon} is true. 
This completes the whole proof of Theorem \ref{T:SComp}.
\end{proof}

\section{Proof of Theorem \ref{T1:Rates}}\label{S:Rates}
% \begin{proof}[Proof of part (1) of Theorem \ref{T1:Rates}]
% \paragraph{Part (1)}
We first prove part (1).
For any compact sets $K\subseteq \R_+^*\times\R$, one can find $\eta>0$, $T>0$ and $N>0$ such that
$K\subseteq [\eta,T]\times [-N,N]$. Then choose $M=2NT/\eta$.
If $\mu(\ud x)=f(x)\ud x$ with $f\in C(\R)$ and $f(x)\ge 0$ for all $x\in\R$, then following the proof of Theorem \ref{T:SComp}, from \eqref{E_:abcup}, we see that
\begin{align*}
  P\left(\inf_{(s,x)\in K}u(s,x)< \beta^{m} \right)&
\le 1- P\left(\cap_{0\le k\le m-1} \left[A_k \cap B_k\right] \right)\\
&\le 2\left[1-(1-c(m))^m\right],
\end{align*} where $c(m)$ is defined in \eqref{E:CM}.
Because $\log(1-x)\ge -2 x$ for $0<x\le 1/2$, when $m$ is sufficiently large, so that
\begin{align}\label{E:1/2}
m \exp\left(- Q \: m^{1-1/a} [\log(m)]^{2-1/a}\right) \le 1/2,
\end{align}
we have that
\[
(1-c(m))^m\ge \exp\left(-2\: m \exp\left(- Q \: m^{1-1/a} [\log(m)]^{2-1/a}\right) \right).
\]
Since $1-e^{-x}\le x$ for $x\geq 0$, if $m$ is sufficiently large such that \eqref{E:1/2} holds, then
\[
\left[1-(1-c(m))^m\right]\le
2\: m \exp\left(- Q \: m^{1-1/a} [\log(m)]^{2-1/a}\right).
\]
If $\mu\in\calM_{a,+}^*(\R)$, then we follow the notation of Case II in the proof of Theorem \ref{T:SComp} with
$\epsilon =\eta/2$. Using the Markov property, for each initial data $u(\epsilon,x,\omega)$, we apply the previous case
to get \eqref{E:Rates1} with  $u(t,x)$ replaced by $V_\omega(t-\epsilon,x)$.
Because the upper bound of which does not depend on $\omega$ and $u(\epsilon,x)$ is independent of $V(t,x)$,
\eqref{E:Rates1} holds for $V(t-\epsilon,x) =u(t,x)$.
This completes the proof of part (1) of Theorem \ref{T1:Rates}.
% \end{proof}

\bigskip
% \section{Proof of Theorem \ref{T:Rates}}\label{S:Rates}
% \begin{proof}[Proof of Theorem \ref{T:Rates}]
% \paragraph{Part (2)}
Now we prove part (2).
Since $f$ is a continuous function, there exists finite constant $c>0$  such that $f(x)\ge c 1_{D}(x)$. Without loss of generality, we assume
that $c=1$.   Let $v(t,x)$ be the solution to \eqref{E:FracHt} with the initial data $1_{D}(x)\ud x$.
By Theorem \ref{T:WComp},  $u(t,x)\ge v(t,x)$ for all $t>0$ and $x\in\R$, a.s.
Hence, it suffices to prove that for all $n\ge 1$,
\[
P\left(\inf_{x\in D\:}\inf_{t\in\:]0,T]} v(t,x)\le  e^{-n}\right) \le A \exp\left(-B(n\log(n))^{(2a-1)/a}\right).
\]
We define a set of $\{\calF_t\}_{t\ge 0}$-stopping times as follows:  $T_0:=0$, and
\[
T_{k+1}:=\inf\left\{s>T_k: \inf_{x\in \:D}v(s,x)\le  e^{-k-1}\right\},
\]
where we use the convention that $\inf \phi=\infty$.
%Then for $x\in\R$ fixed, the process $t \mapsto v(t,x)$ is a $C(\R)$-valued strong Markov process.

Similar to the proof of Theorem \ref{T:SComp},  let $\{\W_k(t,x): k\in\bbN\}$ be time-shifted space-time white noises
and let $v_k(t,x)$ be the unique solution to \eqref{E:FracHt} subject to the noise $\W_k$, starting from $v_k(0,x)=e^{-(k-1)}1_{D}(x)$. Then
\[
w_k(t,x):=e^{k-1}v_k(t,x)
\]
solves
\begin{align*}
% \label{E:FracHt}
 \begin{cases}
  \left(\displaystyle\frac{\partial}{\partial t} - \Dxa \right) w_k(t,x) =
\rho_k\left(w_k(t,x)\right) \W_k(t,x),& t\in \R_+^*:=\;]0,+\infty[\;,\: x\in\R,\cr
w_k(0,x) = 1_{D}(x),
 \end{cases}
\end{align*}
where $\rho_k(x):=e^{k-1}\rho(e^{-(k-1)}x)$.
From the definitions of the stopping times, we  see that
\[
e^{k-1} v(T_{k-1},x)\ge 1_{D}(x),\quad \text{for all $x\in\R$, a.s. on $\{T_{k-1}<\infty\}$, for all $k\ge 1$}.
\]
Therefore, by the strong Markov property and the weak comparison principle in Theorem \ref{T:WComp},
we obtain that on $\{T_{k-1}<\infty\}$,
\begin{equation*}
P\left(T_{k}-T_{k-1} \leq \frac{2t}{n} \Big{|} \calF_{T_{k-1}}\right)\leq P\left(\sup_{(t,x)\in \;]0,2T/n]\times D} \left| w_k(t,x)-w_k(0,x)\right| \geq  1-1/e\right).
\end{equation*}
Since $\rho_k$ is Lipschitz continuous with the same Lipschitz constant as $\rho$, a suitable form of the Kolmogorov continuity theorem (see the arguments in the proof of Lemma \ref{L:LgDiv})
implies that for all $\eta \in\;]0,1-2(a+1)/(p(a-1))[\;$,
there exists a finite constant $Q>0$, not depending on $p$, $n$ and $\tau$, such that for all $p\geq 2, n\geq 1,$ and $\tau \in \;]0,1[\;$,
% \begin{equation*}
%  \sup_{n\geq 1} \E \left\{ \sup_{\substack{x\inD\\ s\in\;]0,\tau[}} \left| \frac{w_n(s,x)-w_n(0,x)}{s^{\eta(a-1)/2a}} \right|^p \right\} \leq c \exp\left(c\,\tau  p^{(2a-1)/(a-1)}\right),
% \end{equation*} which implies
%
\begin{equation}\label{E:HolderTime}
\E \left[ \sup_{(s,x)\in \;]0,\tau]\times D} \left|w_k(s,x)-w_k(0,x)\right|^p \right] \leq Q\, \tau^{p\eta(a-1)/2a}\exp\left(Q\,\tau  p^{(2a-1)/(a-1)}\right).
\end{equation}
Letting $\tau:=2t/n$ for $0<t<T$ and minimizing the right hand side of \eqref{E:HolderTime} over $p$, we obtain that for some finite constant $Q'>0$, not depending on $n$,
\begin{equation*}
P\left(T_{k}-T_{k-1} \leq \frac{2t}{n} \Big{|} \calF_{T_{k-1}}\right)\leq Q' \exp\left\{-Q' \,n^{(a-1)/a}(\log n)^{(2a-1)/a}\right\}.
\end{equation*}
Therefore, we obtain the following:
\begin{align*}
&P\left(\inf_{x\in D\:}\inf_{t\in\:]0,T]} v(t,x)\le  e^{-n}\right)\leq P\{T_n\leq t\}\\
& \leq P\big(\text{at least $\lfloor n/2 \rfloor$-many distinct values $k\in \{1,2, \dots, n\}$ such that $T_k-T_{k-1}\leq 2t/n$} \big)\\
& \leq {n\choose \lfloor n/2 \rfloor}c_1^{\lfloor n/2 \rfloor} \exp\left\{-c_2{\lfloor n/2 \rfloor} \,n^{(a-1)/a}(\log n)^{(2a-1)/a} \right\}.
\end{align*}
This completes the proof of Theorem \ref{T1:Rates}. \myEnd
% \end{proof}

\section{Proof of Theorem \ref{T:Approx}}\label{S:Approx}
\begin{proof}[Proof of Theorem \ref{T:Approx}]
Fix $\epsilon>0$. By Theorems \ref{T:ExUni},  both $u(t,x)$ and $u_\epsilon(t,x)$ are well-defined solutions to \eqref{E:FracHt}.
By Lipschitz continuity of $\rho$ and the moment formulas \eqref{E:MomUp},
\begin{align*}
\Norm{u(t,x)-u_\epsilon(t,x)}_2^2 \le&\quad
\left[\left((\mu \: \psi_\epsilon)*\lMr{\delta}{G}{a}(\epsilon,\cdot)*\lMr{\delta}{G}{a}(t,\cdot)\right)(x)
-\left(\mu*\lMr{\delta}{G}{a}(t,\cdot)\right)(x)\right]^2\\
&+ \LIP_\rho^2 \int_0^t\ud s\int_\R\ud y
\Norm{u(s,y)-u_\epsilon(s,y)}_2^2 \lMr{\delta}{G}{a}^2(t-s,x-y).
\end{align*}
Denote the first part on the above upper bound as $I_\epsilon(t,x)$.
Let $\widetilde{\calK}(t,x):=\calK(t,x;\LIP_\rho)$ and denote $f_\epsilon(t,x):=\Norm{u(t,x)-u_\epsilon(t,x)}_2^2$. Then formally,
\[
(f_\epsilon\star \widetilde{\calK})(t,x) \le \left(I_\epsilon\star \widetilde{\calK}\right)(t,x)
+ \LIP_\rho^2 \left(f_\epsilon\star \lMr{\delta}{G}{a}^2\star \widetilde{\calK}\right)(t,x).
\]
Using the fact that $ \left(\LIP_\rho^2\lMr{\delta}{G}{a}^2\star \widetilde{\calK}\right)(t,x)= \widetilde{\calK}(t,x)-\LIP_\rho^2\lMr{\delta}{G}{a}^2(t,x)$, one has that
\[
\left(f_\epsilon\star\lMr{\delta}{G}{a}^2\right)(t,x)\le \LIP_\rho^{-2} \left(I_\epsilon\star\widetilde{\calK}\right)(t,x).
\]
Hence, it reduces to show that
\begin{align}\label{E_:IKlim0}
\lim_{\epsilon\rightarrow 0} \left(I_\epsilon\star\widetilde{\calK}\right)(t,x) =0,\quad\text{for all $t>0$ and $x\in\R$}.
\end{align}
We first assume that $a\in \:]1,2[\:$. Notice that
\[
I_\epsilon(t,x) =
\left[
\left((\mu\:\psi_\epsilon)*\left[\lMr{\delta}{G}{a}(t+\epsilon,\cdot)-\lMr{\delta}{G}{a}(t,\cdot)\right]\right)(x)
+
\left([\mu\psi_\epsilon-\mu]*\lMr{\delta}{G}{a}(t,\cdot)\right)(x)
\right]^2.
\]
By \cite[(4.3)]{ChenDalang14FracHeat}, for $0<t\le T$ and $x\in\R$,
\begin{align}\label{E_:InitD}
(|\mu \psi_\epsilon| * \lMr{\delta}{G}{a}(t,\cdot))\left(x\right)\le
(|\mu| * \lMr{\delta}{G}{a}(t,\cdot))\left(x\right) \le  C_T t^{-1/a},
\end{align}
with $C_T:=A_a\: K_{a,0}\: (T\vee 1)^{1+1/a}$, where $A_a$ is defined as
\begin{align}
 \label{E:Aa}
A_a:=\sup_{y\in\R}\int_\R \frac{|\mu|(\ud z)}{1+|y-z|^{1+a}}.
\end{align}
Hence, if $0<t +\epsilon\le T$, then
\[
I_\epsilon(t,x)\le 4 C_T t^{-1/a} \left[
\left(|\mu\:\psi_\epsilon|*\left|\lMr{\delta}{G}{a}(t+\epsilon,\cdot)-\lMr{\delta}{G}{a}(t,\cdot)\right|\right)(x)
+
\left(|\mu\psi_\epsilon-\mu|*\lMr{\delta}{G}{a}(t,\cdot)\right)(x)
\right].
\]
Now use the upper bound on $\widetilde{\calK}(t,x)$ in \eqref{E:UpBd-K},
\begin{equation}
 \label{E_:IcalK}
\begin{aligned}
 \left(I_\epsilon\star \widetilde{\calK}\right)(t,x)
\le&  2 C_T C' C_T' \int_0^t \ud s \; s^{-1/a} (t-s)^{-1/a}\left[ g_1(t,s,\epsilon, x)+g_2(t,x,\epsilon,x)\right]
\end{aligned}
\end{equation}
where $C':=C'(a,\delta,\LIP_\rho)$ is defined in Proposition \ref{P:UpperBdd-K},
$C_T' =1+T^{1/a}\exp\left((\LIP^2_\rho\Lambda\Gamma(1/a^*))^{a^*} T\right)$,
$1/a^*+1/a=1$, and
\begin{align*}
 g_1(t,s,\epsilon,x) &=\left(|\mu\:\psi_\epsilon|*\left|\lMr{\delta}{G}{a}(s+\epsilon,\cdot)-\lMr{\delta}{G}{a}(s,\cdot)\right|*\lMr{\delta}{G}{a}(t-s,\cdot)\right)(x) ,\\
g_2(t,s,\epsilon,x) &=\left(|\mu\psi_\epsilon-\mu|*\lMr{\delta}{G}{a}(s,\cdot)*\lMr{\delta}{G}{a}(t-s,\cdot)\right)(x).
\end{align*}
By the semigroup property and the dominated convergence theorem,
\[
g_2(t,s,\epsilon,x) =\left(|\mu\psi_\epsilon-\mu|*\lMr{\delta}{G}{a}(t,\cdot)\right)(x)\rightarrow 0,\quad\text{as $\epsilon\rightarrow 0$}.
\]
Clearly, $g_2(t,s,\epsilon,x)\le 2 (\mu*\lMr{\delta}{G}{a}(t,\cdot))(x)$.
Again, by the dominated convergence theorem and by bounding
$\lMr{\delta}{G}{a}(t+\epsilon,\cdot)$ using \cite[(4.3)]{ChenDalang14FracHeat}, one can show that $\lim_{\epsilon\rightarrow 0}g_1(t,s,\epsilon,x)  =0$.
Then by the semigroup property and \eqref{E_:InitD},
for $0<t+\epsilon \le T$,
\begin{align*}
g_1(t,s,\epsilon,x)\le&  \left(|\mu|*(\lMr{\delta}{G}{a}(s+\epsilon,\cdot)+\lMr{\delta}{G}{a}(s,\cdot))* \lMr{\delta}{G}{a}(t-s,\cdot)\right)(x)\\
=&\left(|\mu|*\lMr{\delta}{G}{a}(t+\epsilon,\cdot)\right) (x)
+\left(|\mu|*\lMr{\delta}{G}{a}(t,\cdot)\right) (x)\\
\le& \:C_T\left((t+\epsilon)^{-1/a}+t^{-1/a}\right)\le 2 C_T t^{-1/a}.
\end{align*}
Hence, both upper bounds on $g_1$ and $g_2$ are integrable over $\ud s$ in \eqref{E_:IcalK}. Therefore, by another application of the dominated convergence theorem, we have proved
\eqref{E_:IKlim0}.
Since both functions $f_\epsilon(t,x)$ and $\lMr{\delta}{G}{a}(t,x)$ are nonnegative and the support of $\lMr{\delta}{G}{a}(t,x)$ is over $\R$, we
can conclude that $\lim_{\epsilon\rightarrow
0}f_\epsilon(t,x) =0$ for almost all $t>0$ and $x\in\R$.

When $a=2$, one can apply the dominated convergence theorem to show that $I_\epsilon(t,x)\rightarrow 0$ as $\epsilon\rightarrow 0$. Another application of the dominated convergence theorem shows that
\eqref{E_:IKlim0} is true. The rest is same as the previous case.
We leave the details for interested readers.
This completes the proof of Theorem \ref{T:Approx}.
% \myCmts{Is there something missing for the passage from almost all $(t,x)$ to all $(t,x)$?}
\end{proof}

\section{Proof of Theorem \ref{T:Holder}}\label{S:Holder}

% \begin{proof}[Proof of Theorem \ref{T2:Holder}]
Without loss of generality, we assume that $\mu\ge 0$.
Let $u(t,x)$ be the solution to \eqref{E:FracHt} starting from $\mu\in\calM_a(\R)$.
Fix $T>0$ and $\epsilon\in\;]0,(T/2)\wedge 1]$.
Denote $V(t,x):=u(t+\epsilon,x)$. By the Markov property, $V(t,x)$ solves \eqref{E:FracHt} with
the time-shifted noise $\W_\epsilon (t,x):=\W(t+\epsilon,x)$ starting from $V(0,x)=u(\epsilon,x)$.
Recall the integral form $ V(t,x)= \widetilde{J}_0(t,x) + \widetilde{I}(t,x)$ in \eqref{E:VInt}.

\paragraph{Time increments}
Recall that $u(t,x)=J_0(t,x)+ I(t,x)$.
Let $0<\epsilon\le t\le t'\le T-\epsilon$.
So
\begin{align*}
\Norm{I(t+\epsilon,x)-I(t'+\epsilon,x)}_p^2
\le& \quad 2\Norm{u(t+\epsilon,x)-u(t'+\epsilon,x)}_p^2\\
&+ 2\left|J_0(t+\epsilon,x)-J_0(t'+\epsilon,x)\right|^2,
\end{align*}
with
\begin{align*}
 \Norm{u(t+\epsilon,x)-u(t'+\epsilon,x)}_p^2 =&
\Norm{V(t,x)-V(t',x)}_p^2\\
\le &
2\Norm{\widetilde{I}(t,x)-\widetilde{I}(t',x)}_p^2 + 2\Norm{\widetilde{J}_0(t,x)-\widetilde{J}_0(t',x)}_p^2.
\end{align*}
Notice that for all $p\ge 2$, by  the Burkholder-Davis-Gundy inequality (see \cite[Lemma 3.3]{ChenDalang13Heat}),
\begin{align*}
  \Norm{\widetilde{I}(t,x)-\widetilde{I}\left(t',x\right)}_p^2
\le  2 z_p^2 \Lip_\rho^2 I_1\left(t,t',x\right)
  +
  2 z_p^2  \Lip_\rho^2 I_2\left(t,t',x\right)\;,
\end{align*}
where $z_p\le 2\sqrt{p}$ and $z_2=1$, and
\begin{gather*}
I_1\left(t,t',x\right)=\iint_{\left[0,t\right]\times\R}\ud s \ud y\:
  \left(\lMr{\delta}{G}{a}\left(t-s,x-y\right)-\lMr{\delta}{G}{a}(t'-s,x-y) \right)^2  \left(\Vip^2+\Norm{V\left(s,y\right)}_p^2\right),\\
I_2\left(t,t',x\right)
  = \iint_{\left[t,t'\right]\times\R}\ud s \ud y\:
\lMr{\delta}{G}{a}^2\left(t'-s,x-y\right)\left(\Vip^2+ \Norm{V\left(s,y\right)}_p^2\right).
% \label{EH:LH2}
\end{gather*}
By part (2) of Lemma \ref{L:MBds}, for some finite constant $Q:=Q(a,\delta,\Lip_\rho,\Vip,\mu,p,\epsilon,T)>0$,
\[
\sup_{(s,y)\in[0,t]\times\R} \Norm{V(s,y)}_p^2
=\sup_{(s,y)\in[\epsilon,t+\epsilon]\times\R} \Norm{u(s,y)}_p^2
\le Q.
\]
Then apply  \cite[Proposition 4.4]{ChenDalang14FracHeat} to see that  for some finite constant $C_1=C_1(a,\delta)>0$,
\begin{align}\label{E:ItidT}
 \Norm{\widetilde{I}(t,x)-\widetilde{I}\left(t',x\right)}_p^2
\le C_1 z_p^2 \Lip_\rho^2 Q
\: |t'-t|^{1-1/a}.
\end{align}

By Minkowski's integral inequality and \eqref{E:timediff} below, for some finite constant $C_2:=C_2(a)>0$,
\begin{align*}
\Norm{\widetilde{J}_0(t,x)-\widetilde{J}_0(t',x)}_p^2
\le & \sup_{y\in\R}\Norm{u(\epsilon,y)}_p^2 \left(\int_\R \ud y \: \left|\lMr{\delta}{G}{a}(t,y)-\lMr{\delta}{G}{a}(t',y)\right|\right)^2\\
\le &  C_2  \: Q\: \left[\log(t'/t)\right]^2\le   C_2 \: t^{-2}  Q\:  |t'-t|^2,
\end{align*}
where in the last step, we have applied the inequality $\log(1+x)\le x$ for $x>-1$.
Because $|t'-t|\le T^{\frac{a+1}{2a}} |t'-t|^{\frac{a-1}{2a}}$ and $t\ge \epsilon$, we have that
\begin{align}\label{E:JTilde}
\Norm{\widetilde{J}_0(t,x)-\widetilde{J}_0(t',x)}_p^2 \le
C_2\: \epsilon^{-2} T^{\frac{a+1}{a}} \: Q\:  |t'-t|^{\frac{a-1}{a}}.
\end{align}
Similarly,
\begin{align*}
\left|J_0(t+\epsilon,x)-J_0(t'+\epsilon,x)\right|^2 \le &\sup_{y\in\R} J_0^2(\epsilon,y)
\left(\int_\R\ud y \left| \lMr{\delta}{G}{a}(t,x)-\lMr{\delta}{G}{a}(t',x)\right|\right)^2\\
\le& C_2\: \epsilon^{-2} \: T^{\frac{a+1}{a}} \: Q\:  |t'-t|^{\frac{a-1}{a}}.
\end{align*}

\paragraph{Space increments}
Fix $t\ge \epsilon$. Let $x, x'\in [-T,T]$. Then
\begin{align*}
\Norm{I(t+\epsilon,x)-I(t+\epsilon,x')}_p^2
\le& \quad 2\Norm{u(t+\epsilon,x)-u(t+\epsilon,x')}_p^2\\
&+ 2\left|J_0(t+\epsilon,x)-J_0(t+\epsilon,x')\right|^2,
\end{align*}
with
\begin{align*}
 \Norm{u(t+\epsilon,x)-u(t+\epsilon,x')}_p^2 \le &
2\Norm{\widetilde{I}(t,x)-\widetilde{I}(t,x')}_p^2 + 2\Norm{\widetilde{J}_0(t,x)-\widetilde{J}_0(t,x')}_p^2.
\end{align*}
For $p\ge 2$, by the Burkholder-Davis-Gundy inequality and \cite[Proposition 4.4]{ChenDalang14FracHeat},
\begin{align*}
  \Norm{\widetilde{I}(t,x)-\widetilde{I}\left(t,x'\right)}_p^2
\le&  2 z_p^2 \Lip_\rho^2 \iint_{\left[0,t\right]\times\R}\ud s \ud y\:
  \left(\lMr{\delta}{G}{a}\left(t-s,x-y\right)-\lMr{\delta}{G}{a}(t-s,x'-y) \right)^2\\
&\hspace{3em}\times\left(\Vip^2+\Norm{V\left(s,y\right)}_p^2\right)\\
\le & 2 z_p^2 \Lip_\rho^2 Q |x'-x|^{a-1}.
\end{align*}
By the Minkowski's integral inequality and \eqref{E:GL1}, for some finite constant $C_3:=C_3(a)>0$,
\begin{align}\notag
\Norm{\widetilde{J}_0(t,x)-\widetilde{J}_0(t,x')}_p^2
\le & \sup_{y\in\R}\Norm{u(\epsilon,y)}_p^2 \left(\int_\R \ud y \: \left|\lMr{\delta}{G}{a}(t,x-y)-\lMr{\delta}{G}{a}(t,x'-y)\right|\right)^2\\ 
\le &  C_3  \: Q\: t^{-1/a}|x'-x|\le  C_3  \: Q\: \epsilon^{-1/a}(2T)^{2-a}|x'-x|^{a-1}.
\label{E:JTildeX}
\end{align}
Similarly,
\begin{align*}
\left|J_0(t+\epsilon,x)-J_0(t+\epsilon,x')\right|^2 \le &\sup_{y\in\R} J_0^2(\epsilon,y)
\left(\int_\R\ud y \left| \lMr{\delta}{G}{a}(t,x)-\lMr{\delta}{G}{a}(t,x')\right|\right)^2\\
\le& C_3  \: Q\: \epsilon^{-1/a}(2T)^{2-a}|x'-x|^{a-1}.
\end{align*}

Finally, combining the two cases, we see that for all compact sets $D\subseteq \R_+^*\times\R$, one can find
$T>0$, $\epsilon\in\;]0,(T/2)\wedge 1]$, such that
$D\subseteq K(\epsilon,T):=[2\epsilon,T]\times[-T,T]$. There is
some finite constant $Q':=Q'(a,\delta,\Lip_\rho,\Vip,\mu,p,\epsilon,T)>0$ such that
for all $(t,x)$ and $(t',x')\in D$,
\begin{align*}
\Norm{I(t,x)-I(t',x')}_p^2\le& Q'\left(|t'-t|^{1-1/a}+|x'-x|^{a-1}\right).
\end{align*}
Then the H\"older continuity follows from Kolmogorov's continuity theorem (see \cite[Theorem 1.4.1]{Kunita90Flow} and \cite[Proposition 4.2]{ChenDalang13Holder}).
Note that $J_0(t,x)$ belongs to $C^{\infty}(\R_+^*\times\R)$ (see \cite[Lemma 4.9]{ChenDalang14FracHeat}).
This completes the proof of Theorem \ref{T:Holder}. \myEnd

\section{Proof of Theorem \ref{T:WeakSol}} \label{S:WeakSol}
The case when $a=2$ is proved in \cite[Proposition 3.4]{ChenDalang13Holder}.
Assume that $1<a<2$. Fix $\phi\in C_c(\R)$. For simplicity, we only prove the case where $\rho(u)=\lambda u$ and $\mu\ge 0$.
As in the proof  \cite[Proposition 3.4]{ChenDalang13Holder},
we only need to prove that
\[
\lim_{t\rightarrow 0_+} \int_\R \ud x \: I(t,x) \phi(x) = 0 \quad\text{in
$L^2(\Omega)$}.
\]
Denote $L(t):=\int_\R I(t,x) \phi(x)\ud x$.
By the stochastic Fubini theorem (see \cite[Theorem 2.6, p. 296]{Walsh86}),
whose assumptions are easily checked,
\[
L(t) =  \int_0^t \int_\R \left(\int_\R \ud x\;  \lMr{\delta}{G}{a}(t-s,x-y) \phi(x)\right)
\rho(u(s,y)) W(\ud s,\ud y).
\]
Hence, by \Itos isometry,
\begin{align*}
\E\left[L(t)^2\right]=
\lambda^2 \int_0^t \ud s \int_\R  \ud y
\left(\int_{\R} \ud x\; \lMr{\delta}{G}{a}(t-s,x-y)\phi(x)\right)^2
\Norm{u(s,y)}_2^2 .
\end{align*}
Assume that $t\le 1$. Since for some constant $C>0$, $|\phi(x)|\le C \lMr{\delta}{G}{a}(1,x)$ for all $x\in\R$, we can apply the semigroup property to get
\begin{align*}
\E\left[L(t)^2\right]\le
C^2\lambda^2\Lambda \int_0^t \ud s \frac{1}{(t+1-s)^{1/a}}
\int_\R  \ud y
\lMr{\delta}{G}{a}(t+1-s,y)
\Norm{u(s,y)}_2^2,
\end{align*}
where the constant $\Lambda$ is defined in \eqref{E:Cst-dLa}.
Apply the moment formula \eqref{E:MomUp}, 
\begin{align*}
\E\left[L(t)^2\right]\le
C^2\lambda^2\Lambda \left[L_1(t)+L_2(t)\right],
\end{align*}
with
\[
L_1(t):= \int_0^t\ud s \frac{1}{(t+1-s)^{1/a}} \int_\R \ud y \: J_0^2(s,y) \lMr{\delta}{G}{a}(t+1-s,y),
\]
and
\[
L_2(t):= \int_0^t\ud s \frac{1}{(t+1-s)^{1/a}} \int_\R \ud y \: \left(J_0^2\star\calK\right)(s,y) \lMr{\delta}{G}{a}(t+1-s,y).
\]
We first consider $L_1(t)$. By \cite[(4.20)]{ChenDalang14FracHeat}, for some constant $C_1:=C_1(a,\delta,\mu)>0$,
$J_0(t,x)\le C_1 t^{-1/a}$. Thus,
\begin{align*}
L_1(t)\le &C_1 \int_0^t\ud s \frac{1}{(t+1-s)^{1/a}s^{1/a}} \int_\R \ud y \: J_0(s,y) \lMr{\delta}{G}{a}(t+1-s,y)\\
=&C_1 J_0(t+1,0)  \int_0^t\ud s \frac{1}{(t+1-s)^{1/a}s^{1/a}} \\
\le &C_1 J_0(t+1,0)  \int_0^t\ud s \frac{1}{(1-s)^{1/a}s^{1/a}} \rightarrow 0, \quad\text{as $t\rightarrow 0$.}
\end{align*}
The case for $L_2(t)$ can be proved in a similar way, where one needs to apply \eqref{E:UpBd-K}. We leave the details for interested readers. This completes the proof of Theorem
\ref{T:WeakSol}. \myEnd

\section*{Appendix}
% abel{S:Appdx}

Recall the kernel function $\tlMr{\epsilon}{\delta}{R}{a}(t,x)$ defined in \eqref{E:R}.
The following two lemmas \ref{L:RGaprxL1} and \ref{L:RGaprx} below show that $\tlMr{\epsilon}{\delta}{R}{a}(t,x)$
is an approximation of $\lMr{\delta}{G}{a}(t,x)$. The proofs of both Lemmas depend on Lemma \ref{L:GH3F3} below.

\begin{lemma}\label{L:GH3F3}
For all $b\in \R$,
\begin{align}\label{E:GH3F3}
\lim_{z\rightarrow\infty} e^{-z} z^{b+1} \sum_{k=1}^\infty \frac{z^{k-1}}{k! \: k^b} =1.
\end{align}
If $b\ge -1$, then
\begin{align}\label{E:GH3F3CB}
C_b:= \sup_{z\ge 0} e^{-z} z^{b+1} \sum_{k=1}^\infty \frac{z^{k-1}}{k! \: k^b}<+\infty.
\end{align}
\end{lemma}

Note that when $b\in\bbN$, the series in \eqref{E:GH3F3} converges to the {\it generalized hypergeometric function} (see \cite[Chapter
16]{NIST2010}):
\[
\sum_{k=1}^\infty\frac{z^{k-1}}{k!\: k^b}=\lMr{b}{F}{b}\Big((\:\underbrace{1,\dots,1}_{\text{$b+1$}}\:),(\:\underbrace{2,\dots,2}_{\text{$b+1$}}\:);z\Big),\quad\text{for $b\in\bbN$ and $z\in \bbC$.}
\]

\begin{proof}[Proof of Lemma \ref{L:GH3F3}]
Clearly, the series converges on $z\in\bbC$ and it defines an entire function.
We first assume that $b\in \bbN$. We will prove \eqref{E:GH3F3} by induction. Clearly, the case $b=0$ is true.
Suppose that \eqref{E:GH3F3} is true for $b$. Now let us consider the case $b+1$: Applying
l'H\^{o}pital's rule and the induction assumption, we obtain
\begin{equation}\label{E:lHopital}
\begin{aligned}
\lim_{z\rightarrow\infty} e^{-z} z^{b+2} \sum_{k=1}^\infty \frac{z^{k-1}}{k! \: k^{b+1}}&=
\lim_{z\rightarrow\infty}\frac{\sum_{k=1}^\infty \frac{z^{k}}{k! \: k^{b+1}}}{\frac{e^z}{z^{b+1}}}
=
\lim_{z\rightarrow\infty}\frac{\sum_{k=1}^\infty \frac{z^{k-1}}{k! \: k^{b}}}{\frac{e^z(z^{b+1}-(b+1)z^b)}{z^{2(b+1)}}}\\
&=
\lim_{z\rightarrow\infty} e^{-z} z^{b+1} \sum_{k=1}^\infty \frac{z^{k-1}}{k! \: k^b}=1.
\end{aligned}
\end{equation}
This proves Lemma \ref{L:GH3F3} for $b\in\bbN$.

Now assume that $1/2 \leq b<3/2$.
Because the function $f(x)=x^b$ for $x\ge 1$ is either concave or convex,
we have that for all $k\ge 1$ and $z\ge 1$,
\[
\left|z^b - k^b\right| \le b \left(z^{b-1}\vee k^{b-1}\right) |k-z|
\le
b \left(z^{b-1}+ k^{b-1}\right) |k-z|.
\]
Hence, for all $z\ge 1$ and $k\ge 1$,
\begin{align}
\left|\frac{1}{k^b}-\frac{1}{z^b}\right| =
\frac{\left|z^b-k^b\right|}{z^b \:k^b} \le
b |k-z|\left(\frac{1}{z\: k^b}+\frac{1}{z^b k}\right)
 \le
b|k-z|\left(\frac{1}{z \: k^{1/2}}+\frac{1}{z^{1/2} k}\right).
\label{E:DeltaB}
\end{align}
Thus, for $z\ge 1$, by Cauchy-Schwartz inequality and \eqref{E:Square},
\begin{align*}
 \left|\sum_{k=2}^\infty \frac{z^k}{k! k^b} - \frac{1}{z^b}\left(e^z-1-z\right)\right|&=
 \left|\sum_{k=2}^\infty \frac{z^k}{k! }\left(\frac{1}{k^{b}}-\frac{1}{z^{b}}\right)\right|\\
&\le
\frac{b}{ z} \sum_{k=2}^\infty \frac{z^k}{k!}\frac{|k-z|}{\sqrt{k}}+ \frac{b}{\sqrt{z}} \sum_{k=2}^\infty \frac{z^k}{k!}\frac{|k-z|}{ k}\\
& \le
\frac{b}{z} \left(\sum_{k=0}^\infty \frac{z^k}{k!}|k-z|^2\right)^{1/2}\left(\sum_{k=1}^\infty \frac{z^k}{k! \:k}\right)^{1/2} \\
& \quad+
\frac{b}{\sqrt{z}} \left(\sum_{k=0}^\infty \frac{z^k}{k!}|k-z|^2\right)^{1/2}\left(\sum_{k=1}^\infty \frac{z^k}{k! \:k^2}\right)^{1/2} \\
&=
\frac{b}{z^{3/2}}\left(e^{z} z^2\sum_{k=1}^\infty \frac{z^{k-1}}{k! \:k}\right)^{1/2}
+
\frac{b}{z^{3/2}}\left(e^{z} z^3\sum_{k=1}^\infty \frac{z^{k-1}}{k! \:k^2}\right)^{1/2}
\end{align*}
Hence, by the previous proof for the case $b\in\mathbb{N}$, and because $b<3/2$,
\[
\lim_{z\rightarrow \infty} z^{b}e^{-z} \left|\sum_{k=2}^\infty \frac{z^k}{k! k^{b}} - \frac{1}{z^{b}}\left(e^z-1-z\right)\right|
\le
\lim_{z\rightarrow \infty}\frac{2 b}{ z^{3/2-b}} =0.
\]
Therefore,
\[
 \lim_{z\rightarrow\infty} z^{b} e^{-z} \sum_{k=2}^\infty \frac{z^k}{k!\; k^{b}}
= \lim_{z\rightarrow\infty} z^{b} e^{-z} \frac{1}{z^{b}}\left(e^z-1-z\right)=1.
\]

Now assume that $b<1/2$. Let $c\in [1/2,3/2[\;$ and $n\in\bbN$ such that $b+n=c$.
Then apply l'H\^{o}pital's rule $n$ times as in \eqref{E:lHopital},
\[
1=\lim_{z\rightarrow\infty} e^{-z} z^{c+1} \sum_{k=1}^\infty \frac{z^{k-1}}{k! \: k^{c}}
=\lim_{z\rightarrow\infty} e^{-z} z^{c} \sum_{k=1}^\infty \frac{z^{k-1}}{k! \: k^{c-1}}
=\cdots
=\lim_{z\rightarrow\infty} e^{-z} z^{c-n+1} \sum_{k=1}^\infty \frac{z^{k-1}}{k! \: k^{c-n}}.
\]
Similarly, if $b\ge 3/2$, then let $c\in [1/2,3/2[\;$ and $n\in\bbN$ such that $b=c+n$.
Then apply l'H\^{o}pital's rule $n$ times as in \eqref{E:lHopital},
\[
\lim_{z\rightarrow\infty} e^{-z} z^{b+1} \sum_{k=1}^\infty \frac{z^{k-1}}{k! \: k^{b}}
=\lim_{z\rightarrow\infty} e^{-z} z^{b} \sum_{k=1}^\infty \frac{z^{k-1}}{k! \: k^{b-1}}
=\cdots
=\lim_{z\rightarrow\infty} e^{-z} z^{b-n+1} \sum_{k=1}^\infty \frac{z^{k-1}}{k! \: k^{b-n}}=1.
\]
This proves \eqref{E:GH3F3} for all $b\in\R$.

Finally, \eqref{E:GH3F3CB} follows from the fact that the function $f(z)=z^{b+1} e^{-z} \sum_{k=1}^\infty \frac{z^{k-1}}{k!\; k^{b}}$ is continuous over $\R_+\cup\{+\infty\}$ with $f(\infty)=1$ and
$|f(0)|<\infty$ if $b\ge -1$ (actually, $f(0)=0$ if $b>-1$ and $f(0)=1$ if $b=-1$).
This completes the proof of Lemma \ref{L:GH3F3}.
\end{proof}

\begin{lemma}\label{L:RGaprxL1}
There exists a finite constant $C>0$ such that
\[
\int_\R \ud x \: \left|\tlMr{\epsilon}{\delta}{R}{a}(t,x)-\lMr{\delta}{G}{a}(t,x)\right|\le
e^{-t/\epsilon}+C\left(\frac{\epsilon}{t}\right)^{1/2},\quad\text{for all $\epsilon>0$ and $t>0$,}
\]
where the constant $C$ can be chosen as
\begin{align}
 \label{E:RGaprxL1}
C=\frac{1}{a}\left(1+K_{a,1}\Gamma\left(\frac{a}{a+2}\right)\Gamma\left(\frac{a+4}{a+2}\right)\right)
 \left[\sup_{z\ge 0}
e^{-z}z (4  z^2+7 z+1) \sum_{k=1}^\infty \frac{z^{k-1}}{k! \: k^2}\right]^{\frac{1}{2}},
\end{align}
with the constant $K_{a,1}$ defined in \cite[(4.3)]{ChenDalang14FracHeat}.
\end{lemma}
\begin{proof}
From \cite[(4.3)]{ChenDalang14FracHeat},
\begin{align}\notag
 \left|\frac{\partial }{\partial t} \lMr{\delta}{G}{a}(t,x)\right| =& \left|-\frac{1}{at}\left(\lMr{\delta}{G}{a}(t,x)+ x \frac{\partial
\lMr{\delta}{G}{a}(t,x)}{\partial x}\right)\right|\\
\notag
&
\le
\frac{1}{at}\left(\lMr{\delta}{G}{a}(t,x)+ \left|x \frac{\partial
\lMr{\delta}{G}{a}(t,x)}{\partial x}\right|\right)\\
\label{E:DtG}
&\le
\frac{1}{at}\left[\lMr{\delta}{G}{a}(t,x)+t^{-\frac{2}{a}}\frac{K_{a,1}|x|}{1+|t^{-1/a} x|^{2+a}}\right].
\end{align}
Thus, for $0<t\le t'$,
\begin{equation}\label{E:timediff}
\begin{aligned}
\int_\R \ud x \left|\lMr{\delta}{G}{a}(t',x)-\lMr{\delta}{G}{a}(t,x)\right| &
\le \int_\R\ud x \int_{t}^{t'}\ud s \left|\frac{\partial }{\partial t} \lMr{\delta}{G}{a}(s,x)\right|\\&
\le \int_t^{t'}\ud s \: \frac{1}{a s} \left(1+\int_\R \ud y \frac{K_{a,1}|y|}{1+|y|^{2+a}}\right)
\le C' \log\left(\frac{t'}{t}\right),
\end{aligned}
\end{equation}
where
\begin{equation}\label{E:constantc'}
C':=\frac{1}{a}\left(1+\int_\R\ud y \frac{K_{a,1}|y|}{1+|y|^{2+a}}\right)=
\frac{1}{a}\left(1+K_{a,1}\Gamma\left(\frac{a}{a+2}\right)\Gamma\left(\frac{a+4}{a+2}\right)\right),
\end{equation}
and the integral in \eqref{E:constantc'} is evaluated by Lemma \ref{L:Beta}.
Notice that
\[
\left|\tlMr{\epsilon}{\delta}{R}{a}(t,x)-\lMr{\delta}{G}{a}(t,x)\right|
\le e^{-t/\epsilon} \lMr{\delta}{G}{a}(t,x)
+ e^{-t/\epsilon} \sum_{k=1}^\infty \left(\frac{t}{\epsilon}\right)^k \frac{1}{k!}
\left|\lMr{\delta}{G}{a}(t,x)-\lMr{\delta}{G}{a}(k\epsilon,x)\right|.
\]
By the above inequality, we have that
\[
\int_\R\ud x\left|\tlMr{\epsilon}{\delta}{R}{a}(t,x)-\lMr{\delta}{G}{a}(t,x)\right|
\le e^{-t/\epsilon} +C' e^{-t/\epsilon}\sum_{k=1}^\infty \left(\frac{t}{\epsilon}\right)^k \frac{1}{k!}
\left|\log (k\epsilon/t)\right|.
\]
Denote the summation over $k$ in above upper bound by $I(t/\epsilon)$.
Because the function $x\mapsto \log(x)$ is concave, $|\log(t'/t)|\le |t'-t|\left(\frac{1}{t'}\vee \frac{1}{t}\right)\le |t'-t|\left(\frac{1}{t'}+ \frac{1}{t}\right)$. So, by letting $z=t/\epsilon$,
\[
I(z) \le \sum_{k=1}^\infty \frac{z^k}{k!}\left|k-z\right|\left(\frac{1}{k}+\frac{1}{z}\right)
=\sum_{k=1}^\infty \frac{z^{k-1}}{k! \: k}\left|k^2-z^2\right|.
\]
Then by Cauchy-Schwartz inequality,
\[
I(z)\le
\left(\sum_{k=1}^\infty \frac{z^{k-1}}{k! \: k^2}\right)^{1/2}
\left(\sum_{k=1}^\infty \frac{z^{k-1}}{k! }\left[k^2-z^2\right]^2\right)^{1/2}.
\]
Notice that
\begin{align}\label{E:ExpZ}
\sum_{k=1}^\infty \frac{z^{k-1}}{k! }\left[k^2-z^2\right]^2=
e^z( 4  z^2+7 z+1)-z^3 \le e^z( 4  z^2+7 z+1).
\end{align}
To prove the equality in \eqref{E:ExpZ}, one can write $k^2= P_k^2+P_k^1$ and
$k^4=P_k^4+6P_k^3+7P_k^2+P_k^1$, where $P_k^n:= k (k-1)\cdots(k-n+1)$.
Hence, $(k^2-z^2)^2=P_k^4+6P_k^3+(7-2 z^2)P_k^2+(1-2z^2)P_k^1+z^4$.
Then use the fact that for $n\ge 1$,
$\sum_{k=1}^\infty \frac{z^{k}}{k!} P_k^n=e^z z^n$.

Therefore,
\[
I(z)e^{-z}\le
\left(e^{-z}z (4  z^2+7 z+1) \sum_{k=1}^\infty \frac{z^{k-1}}{k! \: k^2}\right)^{1/2} \: \frac{1}{\sqrt{z}}.
\]
By Lemma \ref{L:GH3F3}, the function $f(z)=e^{-z}z (4  z^2+7 z+1) \sum_{k=1}^\infty \frac{z^{k-1}}{k! \: k^2}$ is continuous over $\R_+\cup\{+\infty\}$ with $f(0)=0$ and $f(\infty)=4$. Thus,
$\sup_{z\ge 0}f(z)<+\infty$.
This completes the proof of Lemma \ref{L:RGaprxL1}.
\end{proof}

\begin{lemma}\label{L:RGaprx}
We have that
\begin{gather}
\label{E:RGaprx0}
\lim_{\epsilon\rightarrow 0}\tlMr{\epsilon}{\delta}{R}{a}(t,x) = \lMr{\delta}{G}{a}(t,0) \Indt{x= 0},
\\
\label{E:RGaprx1}
\lim_{\epsilon\rightarrow 0}\int_0^t\ud s \int_\R\ud x
\left[\tlMr{\epsilon}{\delta}{R}{a}(s,x) - \lMr{\delta}{G}{a}(s,x)\right]^2=0,
\quad\text{for all $t>0$,}
\end{gather}
and there is a nonnegative constant $C_{a,\delta}<+\infty$ such that
\begin{align}\label{E:RGaprx2}
\int_\R \ud x \tlMr{\epsilon}{\delta}{R}{a}^2(t,x)\le  C_{a,\delta}  \: t^{-\frac{1}{a}},\quad\text{for all $t>0$.}
\end{align}
\end{lemma}
\begin{proof}
Fix $t>0$. Denote $A:= \lMr{\delta}{G}{a}(1,0)$.
Clearly,
\[
\int_0^t\ud s \int_\R\ud x
\left[\tlMr{\epsilon}{\delta}{R}{a}(s,x) - \lMr{\delta}{G}{a}(s,x)\right]^2
= I_1(t,\epsilon) -2 I_2(t,\epsilon) + I_3(t),
\]
where
\begin{align}
I_1(t,\epsilon) &= \iint_{[0,t]\times\R} \ud s\ud x \tlMr{\epsilon}{\delta}{R}{a}^2(s,x), \\
I_2(t,\epsilon) &= \iint_{[0,t]\times\R} \ud s\ud x \tlMr{\epsilon}{\delta}{R}{a}(s,x)\lMr{\delta}{G}{a}(s,x), \\
I_3(t) &= \iint_{[0,t]\times\R} \ud s\ud x \lMr{\delta}{G}{a}^2(s,x).
\end{align}

By the semigroup property and scaling property \cite[(4.1)]{ChenDalang14FracHeat}, we have that
\begin{align}\label{E_:I3}
I_3(t) =\int_0^t \ud s \; \lMr{\delta}{G}{a}(2s,0)=
A \int_0^t (2s)^{-\frac{1}{a}}\ud s = \frac{a A}{2^{1/a}(a-1)}\: t^{1-\frac{1}{a}}.
\end{align}

{\vspace{0.7em}\noindent \bf Step 1.} We first calculate $I_1$. Use the semigroup property and scaling property \cite[(4.1)]{ChenDalang14FracHeat}:
\begin{align*}
I_1(t,\epsilon) &=\int_0^t  \ud s \: e^{-2s/\epsilon}
\sum_{n=1}^\infty
\sum_{m=1}^\infty\left(\frac{s}{\epsilon}\right)^{n+m}\frac{1}{n! m!} \lMr{\delta}{G}{a}((n+m)\epsilon,0)\\
&=
A \int_0^t  \ud s \: e^{-2s/\epsilon}
\sum_{n=1}^\infty
\sum_{m=1}^\infty\left(\frac{s}{\epsilon}\right)^{n+m}\frac{1}{n! m! (n+m)^{1/a} \epsilon^{1/a}}.
\end{align*}
Then by change of variables $u=s/\epsilon$ and let $z=t/\epsilon$, we have that
\[
I_1(t,\epsilon) = A t^{1-\frac{1}{a}}
\frac{1}{z^{1-\frac{1}{a}}}\int_0^z \ud u\:  e^{-2u}\sum_{n,m=1}^\infty\frac{u^{n+m}}{n!m!(n+m)^{1/a}}.
\]
By l'H\^{o}pital's rule,
\[
I_1(t):=\lim_{\epsilon\rightarrow 0}I_1(t,\epsilon)
=  \frac{a A}{a-1} t^{1-\frac{1}{a}}  \lim_{z\rightarrow\infty} z^{1/a} e^{-2z}
\sum_{n,m=1}^\infty\frac{z^{n+m}}{n!m!(n+m)^{1/a}}.
\]
Because $\sum_{n=1}^{k-1}\frac{1}{n!(k-n)!} = \frac{1}{k!}(2^k-2)$,
\[
\sum_{n,m=1}^\infty\frac{z^{n+m}}{n!m!(n+m)^{1/a}}
=\sum_{k=2}^\infty \frac{z^k}{k^{1/a}}\sum_{n=1}^{k-1} \frac{1}{n!(k-n)!}
=\sum_{k=2}^\infty \frac{z^k}{k! k^{1/a} }(2^k-2).
\]
Hence,
\begin{align}\label{E_:I1}
I_1(t)=\frac{a A}{a-1} t^{1-\frac{1}{a}}  \lim_{z\rightarrow\infty}
z^{1/a} e^{-2z} \sum_{k=2}^\infty \frac{z^k}{k! k^{1/a} }(2^k-2)
=\frac{a A}{2^{1/a}(a-1)}\: t^{1-\frac{1}{a}},
\end{align}
where the last equality is due to Lemma \ref{L:GH3F3} with $b=1/a\in [1/2,1]$.

{\vspace{0.7em}\noindent \bf Step 2.} In this step, we calculate $I_2(t):= \lim_{\epsilon\rightarrow\infty} I_2(t,\epsilon)$. Similarly to the Step 1,  use the semigroup property and scaling property
\cite[(4.1)]{ChenDalang14FracHeat}, and then change the variables $u=s/\epsilon$ and $z=t/\epsilon$,
\begin{align*}
I_2(t,\epsilon) &=\int_0^t  \ud s\: e^{-s/\epsilon}\sum_{n=1}^\infty
\left(\frac{s}{\epsilon}\right)^n \frac{1}{n!}\lMr{\delta}{G}{a}(s+n\epsilon,0)= A t^{1-\frac{1}{a}} \frac{1}{z^{1-\frac{1}{a}}} \int_0^z\ud u\; e^{-u} \sum_{n=1}^\infty
\frac{u^n}{n! (u+n)^{1/a}}.
\end{align*}
By l'H\^{o}pital's rule,
\[
I_2(t)=  \frac{a A}{a-1} t^{1-\frac{1}{a}}  \lim_{z\rightarrow\infty} z^{1/a} e^{-z}
\sum_{n=1}^\infty\frac{z^{n}}{n!(z+n)^{1/a}}.
\]
Apply the inequality \eqref{E:DeltaB} below with $b=1/a\in [1/2,1]$,
\[
\left|\frac{1}{(z+n)^{1/a}}-\frac{1}{(2z)^{1/a}}\right|
\le \frac{1}{a} |n-z| \left(\frac{1}{2z \sqrt{n+z}}+\frac{1}{(n+z)\sqrt{2z}}\right)
\le \frac{1+\sqrt{2}}{2a}\:\frac{|n-z|}{z^{3/2}},
\]
for $z\ge 1$ and $n\ge 1$.
Notice that (see the proof of \eqref{E:ExpZ}),
\begin{align}\label{E:Square}
 \sum_{n=0}^\infty \frac{z^n}{n!} |n-z|^2= e^z z.
\end{align}
By Cauchy-Schwartz inequality and \eqref{E:Square}, for $z\ge 1$,
\begin{align*}
 \left|\sum_{n=1}^\infty \frac{z^n}{n! (z+n)^{1/a}} - \frac{1}{(2z)^{1/a}}(e^z-1)\right|
& =
\left|\sum_{n=1}^\infty \frac{z^n}{n!}\left(\frac{1}{(z+n)^{1/a}} - \frac{1}{(2z)^{1/a}}\right)\right|\\
&\le
\frac{1+\sqrt{2}}{2a z^{3/2}}\sum_{n=1}^\infty \frac{z^n}{n!} |n-z|\\
&\le
\frac{1+\sqrt{2}}{2a z^{3/2}}\left(\sum_{n=0}^\infty \frac{z^n}{n!} |n-z|^2\right)^{1/2}\left(\sum_{n=0}^\infty \frac{z^n}{n!}\right)^{1/2}\\
&=
\frac{1+\sqrt{2}}{2a z}e^z.
\end{align*}
Therefore,
\begin{align}
 \label{E_:I2}
I_2(t)=  \frac{a A}{a-1} t^{1-\frac{1}{a}}  \lim_{z\rightarrow\infty} z^{1/a} e^{-z}
 \frac{1}{(2z)^{1/a}}(e^z-1) =
  \frac{a A}{2^{1/a}(a-1)}\: t^{1-\frac{1}{a}}.
\end{align}
Finally, \eqref{E:RGaprx1}
is proved by combining \eqref{E_:I3}, \eqref{E_:I1} and \eqref{E_:I2}.

{\vspace{0.7em}\noindent \bf Step 3.} Now we prove \eqref{E:RGaprx0}.
Clearly, if $x\ne 0$, then $\lim_{\epsilon\rightarrow 0}\tlMr{\epsilon}{\delta}{R}{a}(t,x) =0$. Otherwise, by Lemma \ref{L:GH3F3} with $b=1/a$,
\[
\lim_{\epsilon\rightarrow 0}\tlMr{\epsilon}{\delta}{R}{a}(t,0) =
\frac{A}{t^{1/a}} \lim_{z\rightarrow\infty} z^{1/a}e^{-z}\sum_{k=1}^\infty\frac{z^k}{k!k^{1/a}}=\lMr{\delta}{G}{a}(t,0).
\]

{\vspace{0.7em}\noindent \bf Step 4.} As for \eqref{E:RGaprx2}, denote $I(t;\epsilon)=\int_\R \ud x \tlMr{\epsilon}{\delta}{R}{a}^2(t,x)$.
Following  the arguments in Step 1, 
\begin{align}\label{E_:sup}
I(t,\epsilon)\le \frac{A}{(2t)^{1/a}} \sup_{z\in\R_+}
z^{1/a} e^{-z} \sum_{k=2}^\infty \frac{z^k}{k! k^{1/a} }.
\end{align}
Clearly, the function $f(z)=\sum_{k=2}^\infty \frac{z^k}{k! k^{1/a}}$ is an entire function over $\bbC$. By Lemma \ref{L:GH3F3} and
$\lim_{z\rightarrow 0}z^{1/a}e^{-z}f(z)=0$, we know that the supremum in \eqref{E_:sup}, which depends only on $a$, is finite.
This completes the proof of Lemma \ref{L:RGaprx}
\end{proof}

\begin{lemma}\label{L:MInc}
If $\mu(\ud x) =f(x)\ud x$ with $f\in L^\infty(\R)$, then for all $0<t\le T$ and $x$, $y\in\R$,
\[
\Norm{u(t,x)-u(t,y)}_2^2 \le C t^{-1/a}|x-y|+A_T C_1|x-y|^{a-1} ,
\]
where $K_{a,1}$ is defined in \cite[(4.3)]{ChenDalang14FracHeat}, $C_1:=C_1(a,\delta)$ is defined in \cite[Proposition 4.4]{ChenDalang14FracHeat}, and
\begin{align*}
C&:=8  K_{a,1}\: \Gamma\left(\frac{a+1}{a+2}\right)\Gamma\left(\frac{a+3}{a+2}\right) \sup_{x\in\R}[f(x)]^2,\quad A_T :=\sup_{s\in [0,T]}\sup_{x\in\R} \Norm{\rho(u(s,x))}_2^2.
\end{align*}
\end{lemma}
\begin{proof}
% Without loss of generality, we assume that $x\le y$.
By \Itos isometry,
\begin{multline*}
 \Norm{u(t,x)-u(t,y)}_2^2= \left[J_0(t,x)-J_0(t,y)\right]^2+
\int_0^t\ud s\int_\R \ud z \Norm{\rho(u(s,z))}_2^2 \\
\times \left[\lMr{\delta}{G}{a}(t-s,x-z)-\lMr{\delta}{G}{a}(t-s,y-z)\right]^2.
\end{multline*}
Denote $C_f:=\sup_{x\in\R} |f(x)|$ and fix $0<t<T$. Then
\begin{align*}
\Norm{u(t,x)-u(t,y)}_2^2\le & 2\: C_f^2  \int_\R\ud z \left|\lMr{\delta}{G}{a}(t,x-z)-\lMr{\delta}{G}{a}(t,y-z)\right| \\
& +A_T \int_0^t\ud s \int_\R\ud z \left(\lMr{\delta}{G}{a}(s,x-z)-\lMr{\delta}{G}{a}(s,y-z)\right)^2.
\end{align*}
By \cite[Proposition 4.4]{ChenDalang14FracHeat}, for some constant $C_1:=C_1(a,\delta)$, the second part of the above upper bound is bounded by $A_T C_1 |x-y|^{a-1}$. As
for the first part,
notice that by \cite[(4.3)]{ChenDalang14FracHeat}, for all $t> 0$ and $x$, $y\in\R$,
\begin{align*}
\left|\lMr{\delta}{G}{a}(t,x)-\lMr{\delta}{G}{a}(t,y)\right|
& = \left|\int_x^y \frac{\partial }{\partial x}\lMr{\delta}{G}{a}(t,z) \ud z\right|\le \int_x^y \left|\frac{\partial }{\partial x}\lMr{\delta}{G}{a}(t,z)\right| \ud z\\
&\le K_{a,1} t^{-2/a} \int_x^y  \frac{\ud z}{1+|t^{-1/a}z|^{2+a}}\\
&\le K_{a,1} t^{-2/a} \frac{|x-y|}{1+\left(|t^{-1/a}x|\wedge|t^{-1/a}y|\right)^{2+a}}.
\end{align*}
Thus,
\begin{align*}
 \int_\R\ud z &\left|\lMr{\delta}{G}{a}(t,x-z)-\lMr{\delta}{G}{a}(t,y-z)\right|
 \le K_{a,1} t^{-2/a}  \int_\R\frac{|x-y|\:\ud z}{1+\left(|t^{-1/a}(x-z)|\wedge |t^{-1/a}(y-z)|\right)^{2+a}}\\
& \le K_{a,1} t^{-2/a} |x-y| \int_\R\ud z \left[\frac{1}{1+|t^{-1/a}(x-z)|^{2+a}}+\frac{1}{1+|t^{-1/a}(y-z)|^{2+a}}\right]\\
& = 2 K_{a,1} t^{-2/a} |x-y| \int_\R \frac{\ud z}{1+|t^{-1/a} z|^{2+a}}=
 2 K_{a,1} t^{-1/a} |x-y| \int_\R \frac{\ud z}{1+|z|^{2+a}} .
\end{align*}
Hence, by letting
\[
C_a:=\int_\R \frac{\ud z}{1+|z|^{2+a}} = 2\Gamma\left(\frac{a+1}{a+2}\right)\Gamma\left(\frac{a+3}{a+2}\right),
\]
where the integral is evaluated by Lemma \ref{L:Beta}, we have that
\begin{align}
\label{E:GL1}
  \int_\R\ud z \left|\lMr{\delta}{G}{a}(t,x-z)-\lMr{\delta}{G}{a}(t,y-z)\right| \le 2 K_{a,1} C_a t^{-1/a} |x-y|,\quad\text{for all $x$, $y\in\R$.}
\end{align}
This completes the proof of Lemma \ref{L:MInc}.
\end{proof}
\begin{lemma}\label{L:Beta}
For $a>0$ and $b\in \:]-1,a+1[\:$, $\int_0^\infty \ud y \frac{y^b}{1+y^{2+a}} =\frac{1}{b+1}\Gamma\left(\frac{a-b+1}{a+2}\right)\Gamma\left(\frac{a+b+3}{a+2}\right)$.
\end{lemma}
\begin{proof}
Let $1+y^{2+a}=r^{-1}$. Then $\ud y=-\frac{1}{a+2}r^{-\frac{a+3}{a+2}}(1-r)^{-\frac{a+1}{a+2}}\ud r$. So,
\[
\int_0^1 \ud r\: r \: r^{-\frac{b}{a+2}}(1-r)^{\frac{b}{a+2}} \frac{1}{a+2} r^{-\frac{a+3}{a+2}}(1-r)^{-\frac{a+1}{a+2}}
=
\frac{1}{a+2}\int_0^1 \ud r \: r^{\frac{a+1-b}{a+2}-1} (1-r)^{\frac{b+1}{a+2}-1}.
\]
Then apply the Beta integral and use the recursion $x\Gamma(x)=\Gamma(x+1)$.
\end{proof}

\section*{Acknowledgements}
\addcontentsline{toc}{section}{Acknowledgements}
The authors appreciate many stimulating discussions and supports from Davar Khoshnevisan,
and especially his several suggestions on the proof of the strict comparison principle.
The authors also thank  Carl Mueller for many helpful suggestions.
The authors thank Tom Alberts for some interesting discussions and for his pointing out the reference \cite{Moreno14Pos}.
The first author thanks Robert Dalang and Roger Tribe for many interesting discussions.
% The authors also thank two anonymous referees for a careful
% reading of this paper and many useful suggestions. \cite{TribeZabczyk13notes}

% \newpage
%%%%%%%%%%%%%%%%%%%%%%%%%%%%%%%%%%%%%%%%%%%%%%
%%%%% TAIL: Bibliography, Appendix, CV
%%%%%%%%%%%%%%%%%%%%%%%%%%%%%%%%%%%%%%%%%%%%%%
% \bibliographystyle{alpha}
% \bibliographystyle{abbrv}
% \bibliography{SPDE,myreports}
\def\polhk#1{\setbox0=\hbox{#1}{\ooalign{\hidewidth
  \lower1.5ex\hbox{`}\hidewidth\crcr\unhbox0}}} \def\cprime{$'$}
  \def\cprime{$'$}

\vspace{2em}
\begin{minipage}{0.9\textwidth}
% \small
{\noindent\bf Le Chen} and {\noindent\bf Kunwoo Kim}\\[2pt]
University of Utah\\
Department of Mathematics\\
155 South 1400 East\\
Salt Lake City, Utah, 84112-0090\\
U.S.A.\\
\textit{Emails:} \url{chen, kkim@math.utah.edu}\\
% \textit{URLs:} \url{http://math.utah.edu/~chen (/~kkim)}.
\end{minipage}
% \begin{minipage}{0.4\textwidth}
% {\noindent\bf Kunwoo Kim} \\[2pt]
% Institut de math\'ematiques\\
% \'Ecole Polytechnique F\'ed\'erale\\
% Station 8 \\
% CH-1015 Lausanne\\
% Switzerland\\
% \textit{Email:} \url{robert.dalang@epfl.ch}\\
% \textit{URL:} \url{http://prob.epfl.ch}
% \end{minipage}

\end{document}